\documentclass{amsproc}
\usepackage{amssymb,amsthm,stmaryrd,amsmath,amscd,amsrefs,bbm,todonotes}
\usepackage{tikz-cd}
\usepackage[hidelinks]{hyperref}
\usepackage[english]{babel}
\usepackage{chngcntr}

\counterwithin{equation}{subsection}
\newtheorem{theorem}[equation]{Theorem}
\newtheorem{lemma}[equation]{Lemma}
\newtheorem{proposition}[equation]{Proposition}
\newtheorem{corollary}[equation]{Corollary}
\newtheorem{theoremA}{Theorem}

\newtheorem{corollaryA}{Corollary}

\theoremstyle{definition}

\newtheorem{notation}[equation]{Notation}
\newtheorem{warning}[equation]{Warning}

\newtheorem{construction}[equation]{Construction}
\newtheorem{remark}[equation]{Remark}

\newtheorem{definition}[equation]{Definition}

\title[Inertial and Hodge--Tate weights]{Inertial and Hodge--Tate weights of crystalline representations}
\author{Robin Bartlett}
\email{robinbartlett18@mpim-bonn.mpg.de}
\date{\today}
\begin{document}
\maketitle

\begin{abstract}
	Let $K$ be an unramified extension of $\mathbb{Q}_p$ and $\rho\colon G_K \rightarrow \operatorname{GL}_n(\overline{\mathbb{Z}}_p)$ a crystalline representation. If the Hodge--Tate weights of $\rho$ differ by at most $p$ then we show that these weights are contained in a natural collection of weights depending only on the restriction to inertia of $\overline{\rho} = \rho \otimes_{\overline{\mathbb{Z}}_p} \overline{\mathbb{F}}_p$. Our methods involve the study of a full subcategory of $p$-torsion Breuil--Kisin modules which we view as extending Fontaine--Laffaille theory to filtrations of length $p$.
\end{abstract}

\tableofcontents

\section{Introduction}

Let $K/\mathbb{Q}_p$ be a finite unramified extension with residue field $k$. In this paper we show that if the Hodge--Tate weights of a crystalline representation $\rho$ of $G_K$ are sufficiently small then these weights are encoded in an explicit way by the reduction of $\rho$ modulo~$p$. Using Fontaine--Laffaille theory this is known for Hodge--Tate weights differing by at most $p-1$; we will treat weights differing by at most $p$. Our techniques are local and involve the study of a full subcategory of $p$-torsion Breuil--Kisin modules, which we view as extending ($p$-torsion) Fontaine--Laffaille theory to filtrations of length $p$.

To state our result let $\mathbb{Z}_+^n$ denote the set of $(\lambda_1,\ldots, \lambda_n) \in \mathbb{Z}^n$ with $\lambda_1 \leq \ldots \leq \lambda_n$. In Section~\ref{inert} we show how to attach to any continuous $\overline{\rho}\colon G_K \rightarrow \operatorname{GL}_n(\overline{\mathbb{F}}_p)$ a subset
$$
\operatorname{Inert}(\overline{\rho}) \subset (\mathbb{Z}_+^n)^{\operatorname{Hom}_{\mathbb{F}_p}(k,\overline{\mathbb{F}}_p)}
$$
This subset depends only on the restriction to inertia of the semi-simplification of $\overline{\rho}$, and does so in an explicit fashion. We typically write an element of $\operatorname{Inert}(\overline{\rho})$ as $(\lambda_\tau)_{\tau \in \operatorname{Hom}_{\mathbb{F}_p}(k,\overline{\mathbb{F}}_p)}$ with $\lambda_\tau = (\lambda_{1,\tau} \leq \ldots \leq \lambda_{n,\tau})$.

Throughout Hodge--Tate weights are normalised so that the cyclotomic character has weight $-1$.

\begin{theoremA}\label{theorem}
Let $\rho\colon G_K \rightarrow \operatorname{GL}_n(\overline{\mathbb{Z}}_p)$ be a crystalline representation. For each $\tau \in \operatorname{Hom}_{\mathbb{F}_p}(k,\overline{\mathbb{F}}_p)$ let $\lambda_\tau \in \mathbb{Z}_+^n$ denote the $\tau$-Hodge--Tate weights of $\rho$. If $\lambda_{n,\tau} - \lambda_{1,\tau} \leq p$ for all $\tau$ then
$$
(\lambda_\tau)_\tau \in \operatorname{Inert}(\overline{\rho})
$$
\end{theoremA} 

When $n=2$ and $p>2$ the result is a theorem of Gee--Liu--Savitt \cite{GLS}. When $n=2$ and $p=2$ the result is due to Wang \cite{Wang17}. In this paper we extend their methods to higher dimensions.

As already mentioned, when $\lambda_{n,\tau} - \lambda_{1,\tau} \leq p-1$ the Theorem~\ref{theorem} is a straightforward consequence of Fontaine--Laffaille theory, so the main content of our result is that it applies to Hodge--Tate weights differing by $p$. On the other hand the Theorem~\ref{theorem} does not hold if the condition $\lambda_{n,\tau} - \lambda_{1,\tau} \leq p$ is relaxed. For example, there exist irreducible two dimensional crystalline representations $\rho$ of $G_{\mathbb{Q}_p}$ with Hodge--Tate weights $(-p-1,0)$, whose reduction modulo~$p$ have the form $\overline{\rho} = (\begin{smallmatrix}
\chi_{\operatorname{cyc}} & * \\ 0 & \chi_{\operatorname{cyc}}
\end{smallmatrix})$, see \cite[Th\'eor\`eme 3.2.1]{Berger11}. Here $\chi_{\operatorname{cyc}}$ denotes the cyclotomic character. It is easy to check that $(-p-1,0)$ is not an element of $\operatorname{Inert}(\overline{\rho})$.

Our motivation comes from the weight part of (generalisations of) Serre's modularity conjecture. As a corollary of our result we can prove some new cases of weight elimination for mod~$p$ representations associated to automorphic representations on unitary groups of rank $n$. To be more precise let $F$ be an imaginary CM field in which $p$ is unramified and fix an isomorphism $\iota\colon \overline{\mathbb{Q}}_p \cong \mathbb{C}$. Attached to any RACSDC (regular, algebraic, conjugate self dual, and cuspidal) automorphic representation $\Pi$ of $\operatorname{GL}_n(\mathbb{A}_F)$ there is a continuous irreducible $r_{\iota,p}(\Pi)\colon G_F \rightarrow \operatorname{GL}_n(\overline{\mathbb{Q}}_p)$, cf. the main result of \cite{CH13}. If $\Pi$ is unramified above $p$ then $r_{\iota,p}(\Pi)$ is crystalline above $p$, and if $\lambda = (\lambda_\kappa)_{\kappa} \in (\mathbb{Z}_+^n)^{\operatorname{Hom}(F,\mathbb{C})}$ is the weight of $\Pi$ then the $\kappa$-Hodge--Tate weights\footnote{Using $\iota$ we can identify $\kappa \in \operatorname{Hom}(F,\mathbb{C})$ with pairs $(v,\widetilde{\tau})$ where $v$ is a place of $F$ above $p$ and $\widetilde{\tau} \in \operatorname{Hom}(F_v,\overline{\mathbb{Q}}_p)$. Since $p$ is unramified in $F$, $\widetilde{\tau}$ can be identified with $\tau \in \operatorname{Hom}_{\mathbb{F}_p}(k_v,\overline{\mathbb{F}}_p)$ where $k_v$ denotes the residue field of $F_v$. The $\kappa$-th Hodge--Tate weights of $r_{\iota,\Pi}$ are then the $\tau$-th Hodge--Tate weights of $r_{\iota,p}(\Pi)$ at $v$.} of $r_{\iota,p}(\Pi)$ equal
$$
\lambda_{\kappa} + (0,1,\ldots, n-1)
$$
Therefore, if $\operatorname{W}(\overline{r})^{\operatorname{inert}} \subset (\mathbb{Z}_{+}^n)^{\operatorname{Hom}(F,\mathbb{C})}$ consists of $(\lambda_{\kappa})$ such that $\lambda_{\kappa} + (0,1,\ldots, n-1) \in \operatorname{Inert}(\overline{r}_v)$, Theorem~\ref{theorem} implies
\begin{corollaryA}\label{corollary}
Let $\overline{r}\colon G_F \rightarrow \operatorname{GL}_n(\overline{\mathbb{F}}_p)$ be irreducible and continuous. Let $\operatorname{W}(\overline{r})^{\operatorname{aut}}$ denote the set of weights $\lambda \in (\mathbb{Z}_+^n)^{\operatorname{Hom}(F,\mathbb{C})}$ such that there exists an RACSDC automorphic representation $\Pi$ of $\operatorname{GL}_n(\mathbb{A}_{F})$ which is unramified at $p$, has weight $\lambda$, and is such that $\overline{r}_{\iota,p}(\Pi) \cong \overline{r}$. Then
$$
\operatorname{W}(\overline{r})^{\operatorname{aut}}_{\leq p-n +1} \subset W(\overline{r})^{\operatorname{inert}}_{\leq p-n +1}
$$
where for $* \in \lbrace \operatorname{aut},\operatorname{inert} \rbrace$, $\operatorname{W}(\overline{r})_{\leq p -n + 1}^*$ is the subset containing $(\lambda_{\kappa}) \in \operatorname{W}(\overline{r})^*$ with $\lambda_{n,\kappa} - \lambda_{1,\kappa} \leq p - n + 1$.
\end{corollaryA}

We point out that while the Corollary~\ref{corollary} involves only distinct Hodge--Tate weights, due to the regularity assumptions on our automorphic representations, Theorem~\ref{theorem} does not require such distinctness.

If $\overline{r}$ is assumed to arise from some potentially diagonalisable RACSDC automorphic representation (a notion introduced in \cite{BLGGT}) and if we assume $\overline{r}_v$ is semi-simple for each $v \mid p$ then, under a Taylor--Wiles hypothesis, the inclusion in the Corollary~\ref{corollary} is an equality. This follows from e.g. \cite[Theorem 3.1.3]{BLGG}.

To conclude this introduction we briefly explain our proof of the theorem; let us do this by sketching the content of the various sections in this paper. In the first two sections we recall some basic notions; in Section~\ref{inert} we define the set $\operatorname{Inert}(\overline{\rho})$ and in Section~\ref{filtration} we give some elementary results on filtered modules. In Section~\ref{breuil--kisin} we recall the notion of a Breuil--Kisin module, and recall how to associate to them Galois representations. Breuil--Kisin modules killed by $p$ admit a natural set of weights and in Section~\ref{strong-divisible} we define what it means for a $p$-torsion Breuil--Kisin module to be strongly divisible; it's weights must be contained in $[0,p]$ and a certain explicit condition on its $\varphi$ must be satisfied. We view the category of strongly divisible Breuil--Kisin modules $\operatorname{Mod}^{\operatorname{SD}}_k$ as an extension of $p$-torsion Fontaine--Laffaille theory to filtrations of length $p$. We establish two important properties of $\operatorname{Mod}^{\operatorname{SD}}_k$. The first main property (Proposition~\ref{SDsubquotO}) is shown in Section~\ref{strong-divisible} and states that $\operatorname{Mod}^{\operatorname{SD}}_k$ is stable under subquotients, and that weights behave well along short exact sequences. The second main property (Proposition~\ref{inertweights}) is proved in Section~\ref{irreducible} and concerns the structure of simple objects in $M \in \operatorname{Mod}^{\operatorname{SD}}_k$. We show that for such $M$ the weights of $M$ coincide with the inertial weights of the associated Galois representation. These two properties mirror the situation for Fontaine--Laffaille theory. However, unlike in Fontaine--Laffaille theory, it is not the case that simple $M \in \operatorname{Mod}^{\operatorname{SD}}_k$ are determined by their weights together with their associated Galois representation. This complicates the proofs considerably. Thus, while there are similarities between $\operatorname{Mod}^{\operatorname{SD}}_k$ and Fontaine--Laffaille theory in some respects, the former category is more complicated, reflecting the fact that the reduction of crystalline representations with Hodge--Tate weights in $[0,p]$ is genuinely more subtle than for weights in the Fontaine--Laffaille range. In the final section we recall a theorem of Gee--Liu--Savitt \cite{GLS} which relates $\operatorname{Mod}^{\operatorname{SD}}_k$ with the reduction modulo~$p$ of those crystalline representations with Hodge--Tate weights contained in $[0,p]$. Using this, and the two properties of $\operatorname{Mod}^{\operatorname{SD}}_k$ described above, it is straightforward to deduce Theorem~\ref{theorem}.

\subsection*{Acknowledgements} This paper contains part of the my PhD thesis, and I would like to thank my advisor Fred Diamond for his guidance and support. I would also like to thank Dougal Davis and Misja Steinmetz for helpful conversations, David Savitt for bringing our attention to the work of Wang \cite{Wang17}, and Xavier Caruso for encouraging me to extend the main result (which was originally proved for distinct Hodge--Tate weights). Finally it is a pleasure to acknowledge the debt this paper owes to the work of Toby Gee, Tong Liu, and David Savitt.

\subsection{Notation}\label{notationpart}
Throughout we let $k$ denote a finite field of characteristic $p> 0$ and write $K_0 = W(k)[\frac{1}{p}]$. In the introduction we took $K = K_0$; however some of our constructions are valid for arbitrary finite extensions so now allow $K$ to denote a totally ramified extension of $K_0$ of degree $e$, with ring of integers $\mathcal{O}_K$. At certain points it will be necessary to assume $K = K_0$. 

Let $C$ denote the completion of an algebraic closure $\overline{K}$ of $K$ and let $\mathcal{O}_C$ be its ring of integers, with residue field $\overline{k}$. We write $G_K = \operatorname{Gal}(\overline{K}/K)$ and $v_p$ for the valuation on $C$ normalised so that $v_p(p) =1$. 

We fix a uniformiser $\pi \in K$ and a compatible system $\pi^{1/p^n} \in \overline{K}$ of $p^n$-th roots of $\pi$. Many constructions in this paper depend upon these choices. Set $K_\infty = K(\pi^{1/p^\infty})$ and $G_{K_\infty} = \operatorname{Gal}(\overline{K}/K_\infty)$.

Let $\mu_{p^n}(\overline{K})$ denote the group of $p^n$-th roots of unity in $\overline{K}$ and write $\mathbb{Z}_p(1)$ for the free rank one $\mathbb{Z}_p$-module
$$
\varprojlim \mu_{p^n}(\overline{K})
$$
Let $\chi_{\operatorname{cyc}}\colon G_K \rightarrow \mathbb{Z}_p^\times$ denote the character though which $G_K$ acts on $\mathbb{Z}_p(1)$.

Let $E/\mathbb{Q}_p$ denote a finite extension with ring of integers $\mathcal{O}$ and residue field $\mathbb{F}$. We assume throughout that $K_0 \subset E$. This will be our coefficient field in which the representations we consider will be valued.

If $A$ is any ring of characteristic $p$ we let $\varphi\colon A \rightarrow A$ denote the homomorphism $x \mapsto x^p$. If $A$ is perfect (i.e. $\varphi$ is an automorphism) we let $W(A)$ denote the ring of Witt vectors of $A$ and write $\varphi: W(A) \rightarrow W(A)$ for the automorphism lifting $\varphi$ on $A$.

\section{Inertial weights}\label{inert}

In this section we recall the structure of irreducible torsion representations of $G_K$ and $G_{K_\infty}$. We then define the set $\operatorname{Inert}(\overline{\rho})$ from the introduction.

\subsection{Tame ramification}
Let $K^{\operatorname{ur}}$ and $K^{\operatorname{t}}$ be the maximal unramified and maximal tamely ramified extension of $K$ respectively. Set $I^{\operatorname{t}} = \operatorname{Gal}(K^{\operatorname{t}}/K^{\operatorname{ur}})$. As in \cite[Proposition 2]{Serre72} there is an isomorphism
$$
s\colon I^{\operatorname{t}} \rightarrow \varprojlim l^\times
$$
where in the limit $l$ runs over finite extensions of $k$ with transition maps given by norm maps. This isomorphism sends $\sigma \mapsto (s(\sigma)_{l})_{l}$ where $s(\sigma)_{l}$ is the image in the residue field of $K^{\operatorname{t}}$ of the $\operatorname{Card}(l^\times)$-th root of unity 
$$
\sigma(\pi^{1/\operatorname{Card}(l^\times)})/\pi^{1/\operatorname{Card}(l^\times)}\in K^{\operatorname{t}}
$$
Here $\pi^{1/\operatorname{Card}(l^\times)}$ is any $\operatorname{Card}(l^\times)$-th root of $\pi$; $s(\sigma)_l$ does not depend upon any of these choices. Via $s$ we define the fundamental character
$$
\omega_{l}\colon I^{\operatorname{t}} \rightarrow l^\times
$$
For $\theta \in \operatorname{Hom}_{\mathbb{F}_p}(l,\overline{\mathbb{F}}_p)$ define $\omega_\theta = \theta \circ \omega_l$. Note this is a power of $\omega_l$ and $\omega_{\theta \circ \varphi} = \omega_\theta^p$.
\begin{lemma}\label{uniquertheta}
Any continuous $\chi\colon I^{\operatorname{t}} \rightarrow \overline{\mathbb{F}}_p^\times$ extends to a continuous character of $\operatorname{Gal}(K^{\operatorname{t}}/K)$ if and only if there exist integers $(r_\tau)_{\tau \in \operatorname{Hom}_{\mathbb{F}_p}(k,\overline{\mathbb{F}}_p)}$ such that $\chi = \prod_{\tau} \omega_\tau^{r_\tau}$.
\end{lemma}
\begin{proof}
Since $1 \rightarrow I^{\operatorname{t}} \rightarrow \operatorname{Gal}(K^{\operatorname{t}}/K) \rightarrow G_k \rightarrow 1$ is split, $\chi$ extends to $\operatorname{Gal}(K^{\operatorname{t}}/K)$ if and only if $\chi$ is stable under the conjugation action of $G_k$ on $I^{\operatorname{t}}$. Via $s$ this action is given by the natural action of $G_k$ on $\varprojlim l^\times$, and so $\chi$ extends if and only if $\chi^{p^{[k:\mathbb{F}_p]}} = \chi$. After \cite[Proposition 5]{Serre72} this is equivalent to asking that $\chi$ be a power of $\omega_k$, thus a product as in the lemma.
\end{proof}

In particular we see each $\omega_{l}$ extends to a character of $G_{L}$ where $L/K$ is the unramified extension with residue field $l$. Such an extension is well defined only up to twisting by an unramified character. Our fixed choice of uniformiser $\pi \in K$ allows us to define a canonical choice of extension by sending $\sigma \in G_L$ onto the image in the residue field of the element $\sigma(\pi^{1/\operatorname{Card}(l^\times)}) / \pi^{1/\operatorname{Card}(l^\times)} \in K^{\operatorname{t}}$ where $\pi^{1/\operatorname{Card}(l^\times)}$ is an $\operatorname{Card}(l^\times)$-th root of $\pi$. We shall denote this character again by $\omega_l\colon G_L \rightarrow \overline{\mathbb{F}}_p^\times$. Also, for $\theta \in \operatorname{Hom}_{\mathbb{F}_p}(l,\overline{\mathbb{F}}_p)$ we write $\omega_{\theta} = \theta \circ \omega_{l}$, as characters of $G_L$.

For an extension $L/K$ write $\operatorname{Ind}_L^{K} V$ in place of $\operatorname{Ind}_{\operatorname{Gal}(\overline{K}/L)}^{\operatorname{Gal}(\overline{K}/K)} V$.
\begin{lemma}\label{irred}
If $V$ is a continuous irreducible representation of $G_K$ on a finite dimensional $\overline{\mathbb{F}}_p$-vector space then $V \cong \operatorname{Ind}_L^K \chi$, where $L/K$ is an unramified extension of degree $\operatorname{dim}_{\mathbb{F}} V$ and $\chi\colon G_L \rightarrow \overline{\mathbb{F}}_p^\times$ is a continuous character.
\end{lemma}
\begin{proof}
As $V$ is irreducible the $G_K$-action factors through $G = \operatorname{Gal}(K^{\operatorname{t}}/K)$ by \cite[Proposition 4]{Serre72}. Since $I^{\operatorname{t}}$ is abelian of order prime to $p$, $V|_{I^{\operatorname{t}}}$ is a sum of $\overline{\mathbb{F}}_p^\times$-valued characters. If $\gamma \in G_k$ and $\chi\colon I^{\operatorname{t}} \rightarrow \overline{\mathbb{F}}_p^\times$ is a character define a new character by $\chi^{(\gamma)}(\sigma) = \chi(\gamma^{-1}\sigma\gamma)$. If $I^{\operatorname{t}}$ acts on $v \in V|_{I^{\operatorname{t}}}$ by $\chi$ then $I^{\operatorname{t}}$ acts on $\gamma(v)$ by $\chi^{(\gamma)}$; thus $G_k$ acts on the set of $\chi$ appearing in $V|_{I^{\operatorname{t}}}$. Fix $\chi$ appearing in $V|_{I^{\operatorname{t}}}$ and let $H \subset G$ be the normal subgroup containing $I^{\operatorname{t}}$, corresponding to the stabiliser of $\chi$ in $G_k$. By the orbit-stabiliser theorem $[G:H] \leq \operatorname{dim}_{\overline{\mathbb{F}}_p} V$.

Frobenius reciprocity gives a non-zero map $V|_H \rightarrow \operatorname{Ind}_{I^{\operatorname{t}}}^H \chi$. If $L/K$ is the unramified extension corresponding to $H$ then since the image of $H$ in $G_k$ stabilises $\chi$, this character can be extended to $H$ as in Lemma~\ref{uniquertheta}. Thus $\operatorname{Ind}_{I^{\operatorname{t}}}^H \chi = \chi \otimes \operatorname{Ind}_{I^{\operatorname{t}}}^H \mathbbm{1}$. Since $\operatorname{Ind}_{I^{\operatorname{t}}}^H \mathbbm{1}$ is a discrete $H$-module we can find a finite dimensional sub-representation $R \subset \operatorname{Ind}_{I^{\operatorname{t}}}^H \mathbbm{1}$ so that $V|_H$ is mapped into $\chi \otimes R$. As $\operatorname{Gal}(L^{\operatorname{ur}}/L)$ is abelian $R$ admits a composition series $0 = R_n \subset \ldots \subset R_0 = R$ such that each $R_i/R_{i+1}$ is one-dimensional. If $i$ is the largest integer such that $V|_{H} \rightarrow \operatorname{Ind}^H_{I^{\operatorname{t}}} V$ factors through $\chi \otimes R_i$ then $V|_H \rightarrow \chi \otimes R_i/R_{i+1}$ is non-zero. Frobenius reciprocity gives a non-zero map $V \rightarrow \operatorname{Ind}_L^K (\chi \otimes R_i/R_{i+1})$ which, $V$ being irreducible, is injective. Thus $[G:H] = \operatorname{dim}_{\overline{\mathbb{F}}_p} \operatorname{Ind}_L^K(\chi \otimes R_i/R_{i+1})$ is $\geq \operatorname{dim}_{\overline{\mathbb{F}}_p} V$. The inequality of the first paragraph implies $[G:H] = \operatorname{dim}_{\overline{\mathbb{F}}_p} V$ and so this map is an isomorphism.
\end{proof}
\begin{definition}\label{defofWinert}
Let $\overline{\rho}$ be a continuous representation of $G_K$ on an $n$-dimensional $\overline{\mathbb{F}}_p$-vector space. After Lemma~\ref{irred} there exist continuous characters $\zeta\colon G_{L_\zeta} \rightarrow \overline{\mathbb{F}}_p^\times$ with $L_\zeta / K$ finite unramified, such that
\begin{equation}\label{decomp}
\overline{\rho}^{\operatorname{ss}} \cong \bigoplus_{\zeta} \operatorname{Ind}_{L_\zeta}^K \zeta  
\end{equation}
with each summand irreducible. Let $l_\zeta/k$ denote the residue field of $L_\zeta$. After Lemma~\ref{uniquertheta} there are integers $(r_{\theta,\zeta})_{\theta \in \operatorname{Hom}_{\mathbb{F}_p}(l_\zeta,\overline{\mathbb{F}}_p)}$ such that
$$
\zeta|_{I^{\operatorname{t}}} = \prod \omega_{\theta}^{-r_{\theta,\zeta}}
$$ 
Any such collection of $r_{\theta,\zeta}$ defines a weight $\lambda = (\lambda_\tau)_{\tau \in \operatorname{Hom}_{\mathbb{F}_p}(k,\overline{\mathbb{F}}_p)}$ via $\lambda_\tau = \lbrace r_{\theta,\zeta} \mid \theta|_k = \tau \rbrace$. Define $\operatorname{Inert}(\overline{\rho})$ to be the set of $\lambda$ obtained in this way. 
\end{definition}

It is easy to check that $\operatorname{Inert}(\overline{\rho})$ depends only on $\overline{\rho}^{\operatorname{ss}}|_{I^{\operatorname{t}}}$.

\subsection{$G_{K_\infty}$-representations}
\begin{lemma}\label{intyisom}
Let $K_\infty^{\operatorname{t}} = K_\infty K^{\operatorname{t}}$. Then restriction defines an isomorphism $\operatorname{Gal}(K_\infty^{\operatorname{t}}/K_\infty) \rightarrow \operatorname{Gal}(K^{\operatorname{t}}/K)$. If $L/K$ is a tamely ramified extension this isomorphism identifies $\operatorname{Gal}(L_\infty/K_\infty)$ with $\operatorname{Gal}(L/K)$ where $L_\infty = L K_\infty$.
\end{lemma} 
\begin{proof}
Since $K_\infty/K$ is totally wildly ramified we have $K_\infty \cap K^{\operatorname{t}} = K$. The lemma then follows from Galois theory.
\end{proof}

\begin{corollary}\label{restricttoGKinfty}
	If $V$ is as in Lemma~\ref{irred} then $V|_{G_{K_\infty}} \cong \operatorname{Ind}_{L_\infty}^{K_\infty} \chi|_{G_{L_\infty}}$ where $L_\infty = LK_\infty$.
\end{corollary}

\section{Filtrations}\label{filtration}

This section contains some elementary results on filtered modules; they will be useful later. Consider a commutative ring $A$ and a collection of ideals $(F^iA)_{i \in \mathbb{Z}}$ satisfying
$$
F^{i+1} A \subset F^iA, \qquad (F^i A)(F^j A) \subset F^{i+j}A,\qquad F^i A = A \text{ for $i << 0 $}
$$
Then the category $\operatorname{Fil}(A)$ of filtered $A$-modules consists of $A$-modules $M$ equipped with a collection of $A$-sub-modules $(F^i M)_{i \in \mathbb{Z}}$ satisfying
$$
F^{i+1} M \subset F^i M, \qquad (F^i A)(F^j M) \subset F^{i+j} M,\qquad F^i M = M \text{ for $i << 0$}
$$ 
Morphisms are maps $f\colon M \rightarrow N$ of $A$-modules such that $f(F^iM) \subset F^iN$ for all $i$. If $M$ is an object of $\operatorname{Fil}(A)$ we set $\operatorname{gr}(M) = \bigoplus_i \operatorname{gr}^i (M)$ where $\operatorname{gr}^i (M) = F^i M / F^{i+1}M$. The module $\operatorname{gr}(A)$ admits an obvious structure of a ring and each $\operatorname{gr}(M)$ admits the structure of a module over $\operatorname{gr}(A)$.

\subsection{Strict maps}
If $M$ is an object of $\operatorname{Fil}(A)$ and $N \subset M$ is an $A$-sub-module the induced filtration on $N$ is that given by $F^i N = N \cap F^i M$. If $f\colon M \rightarrow N$ is a surjective $A$-module homomorphism the quotient filtration on $N$ is that given by $F^i N = f(F^i M)$. 

\begin{remark}
For any morphism $f: M \rightarrow N$ in $\operatorname{Fil}(A)$ there is a sequence
$$
\operatorname{ker}(f) \rightarrow M \rightarrow \operatorname{coim}(f) \rightarrow \operatorname{im}(f) \rightarrow N \rightarrow \operatorname{coker}(f)
$$
in $\operatorname{Fil}(A)$. The modules $\operatorname{ker}(f) \subset M$  and $\operatorname{im}(f) \subset N$ are each equipped with the induced filtration. The modules $\operatorname{coker}(f)$ and $\operatorname{coim}(f)$ are equipped with the quotient filtration, coming from $N$ and $M$ respectively.
\end{remark}
\begin{definition}
A morphism $f: M \rightarrow N$ in $\operatorname{Fil}(A)$ is strict if $F^i N \cap f(M) = f(F^i M)$ for all $i \in \mathbb{Z}$. Equivalently $f$ is strict if $\operatorname{coim}(f) \rightarrow \operatorname{im}f$ is an isomorphism in $\operatorname{Fil}(A)$.
\end{definition}

\begin{notation}
The filtration on $A$ induces the structure of a topological ring on $A$; the $F^i A$ form a basis of open neighbourhoods of zero. Similarly the filtration on an object $M$ of $\operatorname{Fil}(A)$ gives $M$ the structure of a topological $A$-module. Then
\begin{itemize}
\item $M$ is discrete if and only if $F^i M = 0$ for $i>>0$;
\item $M$ is Hausdorff if and only if $\cap F^iM = 0$;
\item $M$ is complete if and only if the natural map $M \rightarrow \varprojlim M/F^iM$ is an isomorphism.
\end{itemize}
\end{notation}

\begin{lemma}\label{caenstuff}
Let $f: M \rightarrow N$ be a morphism in $\operatorname{Fil}(A)$ which is an isomorphism of $A$-modules. 
\begin{enumerate}

\item Then $f$ is an isomorphism in $\operatorname{Fil}(A)$ if and only if $\operatorname{gr}^i (M)\rightarrow \operatorname{gr}^i (N)$ is injective for all $i$.
\item If $M$ is complete and $N$ Hausdorff then $f$ is an isomorphism in $\operatorname{Fil}(A)$ if and only if $\operatorname{gr}^i (M) \rightarrow \operatorname{gr}^i (N)$ is surjective for all $i$.
\end{enumerate}
\end{lemma}
\begin{proof}
The following diagram commutes and has exact rows.
$$
\begin{tikzcd}[column sep =small, row sep=small]
0 \arrow{r} & F^{i+1} M \arrow{r} \arrow[hookrightarrow]{d}{a} & F^{i} M \arrow{r} \arrow[hookrightarrow]{d}{b} & \operatorname{gr}^i(M) \arrow{r} \arrow{d}{c} & 0 \\
0 \arrow{r} & F^{i+1} N \arrow{r}  & F^{i} N \arrow{r}  & \operatorname{gr}^i(N) \arrow{r} & 0
\end{tikzcd}
$$
Since $M \rightarrow N$ is an isomorphism of $A$-modules the leftmost and central vertical arrows are injective. For (1) use the snake lemma to obtain an exact sequence $0 \rightarrow \operatorname{ker} c \rightarrow \operatorname{coker}(a) \rightarrow \operatorname{coker}(b) \rightarrow \operatorname{coker}(c)$. One proves $F^i M \rightarrow F^i N$ is surjective by increasing induction on $i$; using as the base case the fact that $F^i M \rightarrow F^iN$ is surjective for $i <<0$, since $F^iM = M$ for $i <<0$. For (2) argue as in \cite[Proposition 6]{SerreLA}.
\end{proof}
\begin{lemma}\label{fantasticstuff}
Let $f: M \rightarrow N$ be a morphism in $\operatorname{Fil}(A)$. Then the following are equivalent.
\begin{enumerate}

\item $f$ is strict;
\item $\operatorname{gr} (\operatorname{ker}(f)) \rightarrow \operatorname{gr}(M) \rightarrow \operatorname{gr} (N)$ is exact;
\item $0 \rightarrow \operatorname{gr} (\operatorname{ker}(f)) \rightarrow \operatorname{gr}(M) \rightarrow \operatorname{gr} (N) \rightarrow \operatorname{gr}(\operatorname{coker}(f)) \rightarrow 0$ is exact.
\end{enumerate}
If $M$ is complete and $N$ is Hausdorff then the same is true with $(2)$ replaced by
\begin{enumerate}
\item[$(2')$] $ \operatorname{gr} (M)\rightarrow \operatorname{gr} (N) \rightarrow \operatorname{gr}(\operatorname{coker}(f))$ is exact for all $i$;
\end{enumerate}
\end{lemma}
\begin{proof}
It is straightforward to check (without any conditions on $M$ and $N$) that (2) is equivalent to $\operatorname{gr}^i\operatorname{coim}(f) \rightarrow \operatorname{gr}^i \operatorname{im}(f)$ being injective for all $i$, that ($2'$) is equivalent to this map being surjective for all $i$, and that (3) is equivalent to this map being an isomorphism for all $i$. Thus $(1) \Leftrightarrow (2) \Leftrightarrow (3)$ follows from Lemma~\ref{caenstuff}(1) applied to the morphism $\operatorname{coim}(f) \rightarrow \operatorname{im}(f)$. Similarly $(1) \Leftrightarrow (2') \Leftrightarrow (3)$ follows from  Lemma~\ref{caenstuff}(2), noting that $M$ being complete implies $\operatorname{coim}(f)$ is complete and $N$ being Hausdorff implies $\operatorname{im}(f)$ is Hausdorff. 
\end{proof}

\begin{corollary}\label{stuffstuffcor}
Let $M$ be a Hausdorff object of $\operatorname{Fil}(A)$ with $A$ complete. Suppose $(m_j)$ is a finite collection of elements of $M$ and suppose that there are integers $r_j$ such that $m_j \in F^{r_j} M$. Let $\overline{m}_j$ denote the image of $m_j$ in $\operatorname{gr}^{r_j}(M)$. If the $\overline{m}_j$ generate $\operatorname{gr}(M)$ over $\operatorname{gr}(A)$ then $M$ is complete and the $m_j$ generate $M$. Further 
$$
F^i M = \sum_j (F^{i-r_j}A)m_j
$$
If the $\overline{m}_j$ form a $\operatorname{gr}(A)$-basis of $\operatorname{gr}(M)$ then the $m_j$ are an $A$-basis of $M$.
\end{corollary}
\begin{proof}
Argue as in \cite[Corollary]{SerreLA} using the second part of Lemma~\ref{fantasticstuff}.
\end{proof}

\subsection{Adapted bases}
We now put ourselves in the following situation. Let $a \in A$ be a nonzerodivisor and equip $A$ with the $a$-adic filtration (so $F^iA = a^iA$). Let $M$ be a finite free $A$-module and let $N \subset M[\frac{1}{a}]$ be a finitely generated $A$-sub-module with $N[\frac{1}{a}] = M[\frac{1}{a}]$. Make $N$ into an object of $\operatorname{Fil}(A)$ by setting $F^i N = a^iM \cap N$.

\begin{lemma}\label{liftinggr}
Suppose that $A$ is complete. Give $N/a$ the quotient filtration and suppose that a finite collection $(g_i)$ of elements of $N$ is given along with integers $(r_i)$ such that $g_i \in F^{r_i} N$. If the images of $g_i$ in $\operatorname{gr}^{r_i}(N/a)$ form a $\operatorname{gr}(A/a) = A/a$-basis of $\operatorname{gr}(N/a)$ then the $(g_i)$ form a basis of $N$ and the $(a^{-r_i}g_i)$ form a basis of $M$.
\end{lemma}
\begin{proof}
The induced filtration on the kernel $aN$ of $N \rightarrow N/a$ is given by $F^i(aN) = aN \cap F^iN = aF^{i-1} N$ (because $a$ is not a zerodivisor). Lemma~\ref{fantasticstuff} implies there is an exact sequence
\begin{equation}\label{exactsequenceusedlater}
0 \rightarrow \operatorname{gr}^{i-1}(N) \xrightarrow{a} \operatorname{gr}^i(N) \rightarrow \operatorname{gr}^i (N/a) \rightarrow 0
\end{equation}
Thus $\operatorname{gr}(N)/a = \operatorname{gr}(N/a)$ where $a \in \operatorname{gr}(A)$ denotes the homogeneous element of degree $1$ represented by $a \in A$. It is then easy to see (e.g. using the graded version of Nakayama's lemma) that the images of the $g_i$ in $\operatorname{gr}(N)$ generate this module over $\operatorname{gr}(A)$. Since $ \cap_i a^i\operatorname{gr}(A) = 0$ they are also $\operatorname{gr}(A)$-linearly independent. As $N$ is finitely generated $N$ is Hausdorf and so we may apply Corollary~\ref{stuffstuffcor} to deduce that the $(g_i)$ form an $A$-basis of $N$ and that
$$
F^n N = \sum (F^{n-r_i}A)g_i
$$
As the $g_i$ are $A$-linearly independent the $(a^{-r_i}g_i)$ are $A$-linearly independent. To show they generate $M$ take $m \in M$ and $n$ large enough that $a^nm \in N$. Then $a^nm \in F^nN$ and so $a^nm = \sum a_i g_i$ with $a_i \in F^{n-r_i}A$. It follows that $m = \sum (a^{r_i-n}a_i) (a^{-r_i}g_i)$ and so, since $(a^{r_i-n}) F^{n-r_i}A \subset A$, we are done.
\end{proof}

\subsection{Filtered vector spaces}
Finally we give criteria to determine when two filtrations on a vector space are the same.
\begin{lemma}\label{equalityisom}
Suppose $A = k$ is a field and let $V$ be an $k$-vector space equipped with two discrete filtrations $G^iV \subset F^i V$. Then 
$$
\sum i \operatorname{dim}_k \operatorname{gr}^i_G(V) \leq \sum i \operatorname{dim}_k \operatorname{gr}^i_F(V)
$$
with equality if and only if $G = F$.
\end{lemma}
\begin{proof}
	Since $\operatorname{dim}_k \operatorname{gr}^i_F(V) = \operatorname{dim}_k F^i V - \operatorname{dim}_k F^{i+1} V$ we have
	$$
	\sum i \operatorname{dim}_k\operatorname{gr}^i_F(V) = \sum \operatorname{dim}_k F^iV
	$$
	Likewise when $F$ is replaced by the filtration $G$. As $G^i V \subset F^i V$, $\operatorname{dim}_k G^i V \leq \operatorname{dim}_k F^iV$; the desired inequality follows. This inequality is an equality if and only if $\operatorname{dim}_k G^iV = \operatorname{dim}_k F^iV$ for all $i$, i.e. if and only if $G = F$.
\end{proof}
\begin{notation}\label{filteredexact}
Say that a sequence of morphisms $M \rightarrow N \rightarrow P$ in $\operatorname{Fil}(A)$ is exact if it is exact as a sequence of $A$-modules and if $M \rightarrow N$ is strict. Lemma~\ref{fantasticstuff} implies that a sequence  $0 \rightarrow M \rightarrow N \rightarrow P \rightarrow 0$ in $\operatorname{Fil}(A)$ which is exact in the category of $A$-modules is exact in $\operatorname{Fil}(A)$ if and only if $0 \rightarrow \operatorname{gr}(M) \rightarrow \operatorname{gr}(N) \rightarrow \operatorname{gr}(P) \rightarrow 0$ is an exact sequence of $A$-modules.
\end{notation}

\begin{corollary}\label{exactsequencesandsums}
Suppose $A = k$ is a field and let $0 \rightarrow M \xrightarrow{f} N \xrightarrow{g} P \rightarrow 0$ be a sequence of finite dimensional discrete objects in $\operatorname{Fil}(k)$ which is exact in the category of $k$-vector spaces. If $f$ (respectively $g$) is strict then 
$$
\sum i \operatorname{dim}_k \operatorname{gr}^i(N) \leq \sum i \operatorname{dim}_k \operatorname{gr}^i(M) + \sum i \operatorname{dim}_k \operatorname{gr}^i(P) \quad \text{(respectively $\geq$)}
$$
Conversely if one of $f$ or $g$ is strict then equality implies the sequence is exact in $\operatorname{Fil}(k)$.
\end{corollary}
\begin{proof}
As $P$ is discrete we can apply Lemma~\ref{equalityisom} to deduce that
$$
\sum i  \operatorname{dim}_k\operatorname{gr}^i(N/M) \leq \sum i \operatorname{dim}_k\operatorname{gr}^i(P)
$$
with equality if and only if $g$ is strict. If $f$ is strict Lemma~\ref{fantasticstuff} tells us that $0 \rightarrow \operatorname{gr}(M) \rightarrow \operatorname{gr}(N) \rightarrow \operatorname{gr}(N/M) \rightarrow 0$ is exact, and so 
$$
\sum i \operatorname{dim}_k\operatorname{gr}^i(N) = \sum i \operatorname{dim}_k\operatorname{gr}^i(M) + \sum i \operatorname{dim}_k\operatorname{gr}^i(N/M)
$$
The lemma follows when we assume $f$ is strict. If $g$ is strict one argues similarly, applying Lemma~\ref{equalityisom} to the map $M \rightarrow \operatorname{ker}(g)$.
\end{proof}

\section{Breuil--Kisin modules}\label{breuil--kisin}

\subsection{Etale $\varphi$-modules}
First we recall the description of $G_{K_\infty}$-representations given by etale $\varphi$-modules.

\begin{definition}
	Let $\mathcal{O}_{C^\flat}$ be the inverse limit of the system 
	$$
	  \mathcal{O}_C/p \leftarrow  \mathcal{O}_C/p \leftarrow  \mathcal{O}_C/p  \leftarrow \ldots
	$$
	with transition maps $x \mapsto x^p$. This is a perfect integrally closed ring of characteristic $p$. There is a multiplicative identification $\mathcal{O}_{C^\flat} = \varprojlim \mathcal{O}_C$ (the limit again taken with respect to the transition maps $x \mapsto x^p$) given by 
	$$
	(\overline{x}_n)_n \mapsto \Big(\lim_{m \rightarrow \infty} x_{m+n}^{p^{m}}\Big)_n
	$$
	where $x_m \in \mathcal{O}_C$ is any lift of $\overline{x}_m$. We write $x \mapsto x^\sharp$ for the projection onto the first coordinate $\mathcal{O}_{C^\flat} \rightarrow \mathcal{O}_C$. Let $C^\flat$ denote the field of fractions of $\mathcal{O}_{C^\flat}$. The formula $v^\flat(x) = v_p(x^\sharp)$ defines a valuation on $C^\flat$ for which it is complete. The field $C^\flat$ is also algebraically closed. Further, the action of $G_K$ on $\mathcal{O}_C$ induces a continuous action of $G_K$ on $\mathcal{O}_{C^\flat}$ and $C^\flat$. 
	\end{definition}

\begin{notation}\label{firstrings}
Let $\mathfrak{S} = W(k)[[u]]$ and $A_{\operatorname{inf}} = W(\mathcal{O}_{C^\flat})$. Both rings are equipped with a $\mathbb{Z}_p$-linear endomorphism $\varphi$; on $A_{\operatorname{inf}}$ this is the usual Witt vector Frobenius and on $\mathfrak{S}$ it is given by $\sum a_iu^i \mapsto \sum \varphi(a_i)u^{ip}$. The system $\pi^{1/p^n}$ defines an element $\pi^\flat = (\pi,\pi^{1/p},\ldots) \in \mathcal{O}_{C^\flat}$ and we embed $\mathfrak{S} \rightarrow A_{\operatorname{inf}}$ by mapping $u \mapsto [\pi^\flat]$ (where $[\cdot]$ denotes the Teichmuller lifting). This embedding is compatible with $\varphi$. Let $\mathcal{O}_{\mathcal{E}}$ denote the $p$-adic completion of $\mathfrak{S}[\frac{1}{u}]$. Then $\varphi$ on $\mathfrak{S}$ extends to $\mathcal{O}_{\mathcal{E}}$ and the embedding $\mathfrak{S} \rightarrow A_{\operatorname{inf}}$ extends to a $\varphi$-equivariant embedding $\mathcal{O}_{\mathcal{E}} \rightarrow W(C^\flat)$.
\end{notation}

By functoriality there are $\varphi$-equivariant $G_K$-actions on $A_{\operatorname{inf}} = W(\mathcal{O}_{C^\flat})$ and $W(C^\flat)$ lifting those modulo~$p$.
\begin{definition}\label{etphimod}
An etale $\varphi$-module is a finitely generated $\mathcal{O}_{\mathcal{E}}$-module $M^{\operatorname{et}}$ equipped with an isomorphism
$$
\varphi_{M^{\operatorname{et}}}\colon M^{\operatorname{et}} \otimes_{\mathcal{O}_{\mathcal{E}},\varphi} \mathcal{O}_{\mathcal{E}} \xrightarrow{\sim} M^{\operatorname{et}}
$$
We may interpret $\varphi_{M^{\operatorname{et}}}$ as a $\varphi$-semilinear map $M^{\operatorname{et}} \rightarrow M^{\operatorname{et}}$ via $m \mapsto \varphi_{M^{\operatorname{et}}}(m \otimes 1)$. When there is no risk of confusion we shall write $\varphi$ in place of $\varphi_{M^{\operatorname{et}}}$. Let $\operatorname{Mod}^{\operatorname{et}}_K$ denote the abelian category of etale $\varphi$-modules.
\end{definition}
\begin{construction}
Since the action of $G_{K_\infty}$ on $C^\flat$ fixes $\pi^\flat$ the $\mathbb{Z}_p$-module
$$
T(M^{\operatorname{et}}) = (M^{\operatorname{et}} \otimes_{\mathcal{O}_{\mathcal{E}}} W(C^\flat))^{\varphi =1}
$$
admits a $\mathbb{Z}_p$-linear action of $G_{K_\infty}$ (given by the trivial action on $M^{\operatorname{et}}$ and natural $G_{K_\infty}$-action on $W(C^\flat)$). This describes a functor from $\operatorname{Mod}^{\operatorname{et}}_K$ to the category of finitely generated $\mathbb{Z}_p$-modules equipped with a continuous $\mathbb{Z}_p$-linear $G_{K_\infty}$-action.
\end{construction}

\begin{proposition}[Fontaine]\label{font}
	The functor $M^{\operatorname{et}} \mapsto T(M^{\operatorname{et}})$ is an exact equivalence of categories. The representation $T(M^{\operatorname{et}})$ is determined up to isomorphism by the existence of a $\varphi,G_{K_\infty}$-equivariant identification
	$$
	M^{\operatorname{et}} \otimes_{\mathcal{O}_{\mathcal{E}}} W(C^\flat) = T(M^{\operatorname{et}}) \otimes_{\mathbb{Z}_p} W(C^\flat)
	$$
\end{proposition}
\begin{proof}
The embedding $\mathcal{O}_{\mathcal{E}} \rightarrow W(C^\flat)$ reduces modulo~$p$ to an inclusion of $k((u))$ in $C^\flat$. The completion of $K_\infty$ is a perfectoid field in the sense of \cite{Sch12}, whose tilt is the completed perfection of $k((u)) \subset C^\flat$.  It follows from \cite[Theorem 3.7]{Sch12} that the action of $G_{K_\infty}$ on $C^\flat$ identifies $G_K = G_{k((u))}$. Let $\mathcal{O}_{\widehat{\mathcal{E}^{\operatorname{ur}}}}$ be the $p$-adic completion of the Cohen ring (i.e. the discrete valuation ring of characteristic zero with uniformizer $p$) with residue field $k((u))^{\operatorname{sep}}$. Then $\mathcal{O}_{\widehat{\mathcal{E}^{\operatorname{ur}}}}$ may be identified as a subring of $W(C^\flat)$ stable under the action of $G_{K_\infty}$ and $\varphi$. The proposition with $T(M^{\operatorname{et}})$ replaced by $T'(M^{\operatorname{et}}) := (M^{\operatorname{et}} \otimes_{\mathcal{O}_{\mathcal{E}}} \mathcal{O}_{\widehat{\mathcal{E}^{\operatorname{ur}}}})^{\varphi =1}$ follows from \cite[Proposition 1.2.6]{Fon00} applied with $E = k((u))$. It therefore suffices to show the inclusion $T'(M^{\operatorname{et}}) \subset T(M^{\operatorname{et}})$ is an equality. Since we know there are $\varphi$-equivariant identifications
$$
M^{\operatorname{et}} \otimes_{\mathcal{O}_{\mathcal{E}}} W(C^\flat) = T'(M^{\operatorname{et}}) \otimes_{\mathbb{Z}_p} W(C^\flat)
$$
the equality follows by taking $\varphi$-invariants.
\end{proof}

\subsection{Breuil--Kisin modules}
Breuil--Kisin modules appear as special $\mathfrak{S}$-lattices inside etale $\varphi$-modules.

\begin{definition}
	A Breuil--Kisin module is a finitely generated $\mathfrak{S}$-module $M$ equipped with an isomorphism
	$$
	\varphi_M\colon M \otimes_{\mathfrak{S},\varphi} \mathfrak{S}[\tfrac{1}{E}] \xrightarrow{\sim} M[\tfrac{1}{E}]
	$$
	Here $E(u) \in \mathfrak{S}$ denotes the minimal polynomial of $\pi$ over $K_0$. We may interpret $\varphi_M$ as a $\varphi$-semilinear map $M \mapsto M[\frac{1}{E}]$ via $m \mapsto \varphi_M(m \otimes 1)$. When there is no risk of confusion we write $\varphi$ in place of $\varphi_M$. Let $\operatorname{Mod}^{\operatorname{BK}}_K$ denote the abelian category of Breuil--Kisin modules.
\end{definition}
\begin{notation}\label{M^phi}
If $M \in \operatorname{Mod}^{\operatorname{BK}}_K$ we write $M^\varphi \subset M[\frac{1}{E}]$ for the image of 
$$
M \rightarrow M \otimes_{\varphi,\mathfrak{S}} \mathfrak{S}[\tfrac{1}{E}] \xrightarrow{\varphi_M} M[\tfrac{1}{E}]
$$ 
More generally we use this notation whenever $A$ is any ring equipped with a Frobenius $\varphi$ and $M$ is an $A$-module equipped with a map $\varphi_M: M \otimes_{\varphi,A} A[\tfrac{1}{a}] \rightarrow M[\tfrac{1}{a}]$ for some $a \in A$. Then $M^\varphi := \varphi_M(M \otimes 1) \subset M[\frac{1}{a}]$.
\end{notation}

\begin{construction}
Note $E(u)$ is a unit in $\mathcal{O}_{\mathcal{E}}$. Thus if $M \in \operatorname{Mod}^{\operatorname{BK}}_K$ then $M \otimes_{\mathfrak{S}} \mathcal{O}_{\mathcal{E}}$ is an etale $\varphi$-module and
$$
T(M) := T(M \otimes_{\mathfrak{S}} \mathcal{O}_{\mathcal{E}}) = (M \otimes_{\mathfrak{S}} W(C^\flat))^{\varphi =1}
$$
defines a functor from $\operatorname{Mod}^{\operatorname{BK}}_K$ to the category of continuous $G_{K_\infty}$-representations on finitely generated $\mathbb{Z}_p$-modules. Since $\mathfrak{S} \rightarrow \mathcal{O}_{\mathcal{E}}$ is flat Proposition~\ref{font} implies $M \mapsto T(M)$ is exact on $\operatorname{Mod}^{\operatorname{BK}}_K$.
\end{construction}

\begin{remark}\label{kisinff}
Kisin \cite[Proposition 2.1.12]{Kis06} has shown $M \mapsto T(M)$ is fully faithful when restricted to Breuil--Kisin modules which are free over $\mathfrak{S}$. However if one does not restrict to Breuil--Kisin modules which are free over $\mathfrak{S}$ then this is not true.
\end{remark}
\begin{construction}\label{internalhom}
For $M,N \in \operatorname{Mod}^{\operatorname{BK}}_K$ the $\mathfrak{S}$-module
$$
\operatorname{Hom}(M,N) := \operatorname{Hom}_{\mathfrak{S} }(M,N)
$$
of $\mathfrak{S}$-linear homomorphisms $M \rightarrow N$ is made into an object of $\operatorname{Mod}^{\operatorname{BK}}_K(\mathcal{O})$ as follows. Since $\varphi \colon \mathfrak{S} \rightarrow \mathfrak{S}$ is flat the natural map $\operatorname{Hom}_{\mathfrak{S}}(M,N)^{\mathcal{O}} \otimes_{\varphi} \mathfrak{S}[\frac{1}{E}] \rightarrow \operatorname{Hom}_{\mathfrak{S}[\frac{1}{E}]}(M \otimes_{\varphi} \mathfrak{S}[\frac{1}{E}],N \otimes_{\varphi} \mathfrak{S}[\frac{1}{E}])$ is an isomorphism. Similarly the natural map $\operatorname{Hom}_{\mathfrak{S}}(M,N)[\frac{1}{E}] \rightarrow \operatorname{Hom}_{\mathfrak{S}[\frac{1}{E}]}(M[\frac{1}{E}],N[\frac{1}{E}])$ is an isomorphism. As such the isomorphism 
$$
\operatorname{Hom}_{\mathfrak{S}[\frac{1}{E}]}(M \otimes_{\varphi} \mathfrak{S}[\tfrac{1}{E}],N \otimes_{\varphi} \mathfrak{S}[\tfrac{1}{E}]) \rightarrow \operatorname{Hom}_{\mathfrak{S}[\frac{1}{E}]}(M[\tfrac{1}{E}],N[\tfrac{1}{E}])
$$
given by $f \mapsto \varphi_N \circ f \circ \varphi_M^{-1}$ makes $\operatorname{Hom}(M,N)$ into a Breuil--Kisin module. Note that
$$
T(\operatorname{Hom}(M,N)) = \operatorname{Hom}_{\mathbb{Z}_p}(T(M),T(N))
$$
as $G_{K_\infty}$-representations, where the $G_{K_\infty}$-action on the right is via $\sigma(f) = \sigma \circ f \circ \sigma^{-1}$.
\end{construction} 

\subsection{Coefficients}
In practice we are interested in representations valued in extensions of $\mathbb{Z}_p$. For this reason we introduce a variant of $\operatorname{Mod}^{\operatorname{BK}}_K$. 

\begin{definition}
Recall the $\mathbb{Z}_p$-algebra $\mathcal{O}$ defined in Subsection~\ref{notationpart}. A Breuil--Kisin module with $\mathcal{O}$-action is a pair $(M,\iota)$ where $M \in \operatorname{Mod}^{\operatorname{BK}}_K$ and $\iota$ is a $\mathbb{Z}_p$-algebra homomorphism $\iota\colon \mathcal{O} \rightarrow \operatorname{End}_{\operatorname{BK}}(M)$. Equivalently a Breuil--Kisin module with $\mathcal{O}$-action is an $\mathfrak{S}_{\mathcal{O}} = \mathfrak{S} \otimes_{\mathbb{Z}_p} \mathcal{O}$-module $M$ equipped with an isomorphism
$$
M \otimes_{\varphi,\mathfrak{S}_{\mathcal{O}}} \mathfrak{S}_{\mathcal{O}}[\tfrac{1}{E}] \xrightarrow{\sim} M[\tfrac{1}{E}]
$$
Here $\varphi$ on $\mathfrak{S}_{\mathcal{O}}$ denotes the $\mathcal{O}$-linear extension of $\varphi$ on $\mathfrak{S}$. Let $\operatorname{Mod}^{\operatorname{BK}}_K(\mathcal{O})$ denote the category of Breuil--Kisin modules with $\mathcal{O}$-action.
\end{definition}

\begin{remark}
By functoriality $M \mapsto T(M)$ induces an exact functor from $\operatorname{Mod}^{\operatorname{BK}}_K(\mathcal{O})$ into the category of continuous representations of $G_{K_\infty}$ on finitely generated $\mathcal{O}$-modules.
\end{remark}

\begin{construction}\label{Ohom}
Let $M,N \in \operatorname{Mod}^{\operatorname{BK}}_K(\mathcal{O})$. Then
$$
\operatorname{Hom}(M,N)^{\mathcal{O}}:= \operatorname{Hom}_{\mathfrak{S} \otimes_{\mathbb{Z}_p} \mathcal{O}}(M,N)
$$
is made into an object of $\operatorname{Mod}^{\operatorname{BK}}_K(\mathcal{O})$ as in Construction~\ref{internalhom}. Again we have
$$
T(\operatorname{Hom}(M,N)^{\mathcal{O}}) = \operatorname{Hom}_{\mathcal{O}}(T(M),T(N))
$$
as $G_{K_\infty}$-representations.
\end{construction}

\begin{construction}\label{O-semilinear}
	The embedding $\mathcal{O}[u] \rightarrow \mathfrak{S} \otimes_{\mathbb{Z}_p} \mathcal{O}$ given by $\sum a_iu^i \mapsto \sum u^i \otimes a_i$ extends by continuity to an embedding $\mathcal{O}[[u]] \rightarrow \mathfrak{S} \otimes_{\mathbb{Z}_p} \mathcal{O}$. Recall that $K_0 \subset E$ by assumption so that the map
	$$
	(\sum a_i u^i) \otimes b \mapsto (\sum \tau(a_i)bu^i)_\tau
	$$
	describes an isomorphism of $\mathcal{O}[[u]]$-algebras $\mathfrak{S} \otimes_{\mathbb{Z}_p} \mathcal{O} \rightarrow \prod_{\tau} \mathcal{O}[[u]]$, the product running over $\tau \in \operatorname{Hom}_{\mathbb{F}_p}(k,\mathbb{F})$ (we abusively write $\tau$ also for its extension to an embedding $\tau\colon W(k) \rightarrow \mathcal{O}$). Let $\widetilde{e}_\tau \in \mathfrak{S} \otimes_{\mathbb{Z}_p} \mathcal{O}$ be the idempotent corresponding to $\tau$. As $\widetilde{e}_\tau$ is determined by the property $(a \otimes 1)\widetilde{e}_\tau = (1 \otimes \tau(a))\widetilde{e}_\tau$ for $a \in W(k)$, the map $\varphi \otimes 1$ sends
	$$
	\widetilde{e}_{\tau \circ \varphi} \mapsto \widetilde{e}_\tau
	$$
	If $M \in \operatorname{Mod}^{\operatorname{BK}}_K(\mathcal{O})$ we set $M_\tau = \widetilde{e}_\tau M$ which we view as an $\mathcal{O}[[u]]$-algebra. By the above $\varphi_M$ restricts to a map
	\begin{equation}\label{Olalal}
	M_{\tau \circ \varphi} \otimes_{\varphi,\mathcal{O}[[u]]} \mathcal{O}[[u]] \rightarrow M_\tau[\tfrac{1}{\tau(E)}]
	\end{equation}
	which becomes an isomorphism after inverting $\tau(E)$. Here $\varphi$ on $\mathcal{O}[[u]]$ is that induced by $\varphi \otimes 1$ on $\mathfrak{S} \otimes_{\mathbb{Z}_p} \mathcal{O}$, i.e. is given by $\sum a_i u^i \mapsto \sum a_iu^{ip}$.
\end{construction}

\begin{corollary}\label{freethings}
	\begin{enumerate}
		\item If $M \in \operatorname{Mod}^{\operatorname{BK}}_K(\mathcal{O})$ is free as an $\mathfrak{S}$-module then it is free as an $\mathfrak{S} \otimes_{\mathbb{Z}_p} \mathcal{O}$-module. 
		\item Let $\varpi \in \mathcal{O}$ be a uniformiser and suppose $M \in \operatorname{Mod}^{\operatorname{BK}}_K(\mathcal{O})$ is $\varpi$-torsion. If $M$ is free as an $\mathfrak{S} / p = k[[u]]$-module then it is free as a module over $k[[u]] \otimes_{\mathbb{F}_p} \mathbb{F}$.
	\end{enumerate}	
\end{corollary}
\begin{proof}
If $M$ is free over $\mathfrak{S}$ then each $M_\tau$ is free over $\mathcal{O}[[u]]$. By \eqref{Olalal} the rank of $M_\tau$ over $\mathcal{O}[[u]]$ does not depend on $\tau$ so $M = \prod_\tau M_\tau$ is free over $\mathfrak{S} \otimes_{\mathbb{Z}_p} \mathcal{O}$. (2) follows similarly.
\end{proof}

\section{Strongly divisibility}\label{strong-divisible}
\subsection{Torsion Breuil--Kisin modules}

\begin{definition}
	Denote by $\operatorname{Mod}^{\operatorname{BK}}_{k} \subset \operatorname{Mod}^{\operatorname{BK}}_K$ the full subcategory whose objects are modules which are free over $\mathfrak{S}/ p = k[[u]]$. 
\end{definition}

\begin{remark}\label{ptorsioniseasy}
	An $M \in \operatorname{Mod}^{\operatorname{BK}}_k$ is the same thing as a $k[[u]]$-lattice inside an etale $\varphi$-module over $\mathcal{O}_{\mathcal{E}}/p = k((u))$ because $E(u) \equiv u^e$ modulo~$p$.\footnote{ In particular there are many $p$-torsion Breuil--Kisin modules giving rise to the same etale $\varphi$-module. This is in contrast to the integral situation, see Remark~\ref{kisinff}.}
\end{remark}

\begin{lemma}\label{ptorsionexactness}
	The functor $M \mapsto T(M)$ restricts to an essentially surjective functor from $\operatorname{Mod}^{\operatorname{BK}}_{k}$ to the category of continuous representations of $G_{K_\infty}$ on finite dimensional $\mathbb{F}_p$-vector spaces. If $M \in \operatorname{Mod}^{\operatorname{BK}}_{k}$ and 
	$$
	0 \rightarrow T_1 \rightarrow T(M) \rightarrow T_2 \rightarrow 0
	$$
	is an exact sequence of $G_{K_\infty}$-representations then there exists a unique exact sequence 
	$$
	0 \rightarrow M_1 \rightarrow M \rightarrow M_2 \rightarrow 0
	$$
	in $\operatorname{Mod}^{\operatorname{BK}}_{k}$ such that $T(M_i) = T_i$.
\end{lemma}
\begin{proof}
	If $T$ is an $\mathbb{F}_p$-representation of $G_{K_\infty}$ then there exists a $p$-torsion $M^{\operatorname{et}} \in \operatorname{Mod}^{\operatorname{et}}_K$ such that $T(M^{\operatorname{et}}) = T$. Remark~\ref{ptorsioniseasy} shows that any $k[[u]]$-lattice $M \subset M^{\operatorname{et}}$ is an object of $\operatorname{Mod}^{\operatorname{BK}}_k$ with $T(M) = T$.
	
	For the second part, there exists an exact sequence $0 \rightarrow M_1^{\operatorname{et}} \rightarrow M^{\operatorname{et}} \rightarrow M_2^{\operatorname{et}} \rightarrow 0$ such that $T(M_i^{\operatorname{et}}) = T_i$ and such that $M^{\operatorname{et}} = M[\frac{1}{u}]$. Since $M_2$ is torsion-free we must have $M_1 = M \cap M_1^{\operatorname{et}}$ and $M_2 = \operatorname{Im}(M) \cap M_2^{\operatorname{et}}$.
\end{proof}

\begin{construction}\label{compositionseries}
	Let $M \in \operatorname{Mod}^{\operatorname{BK}}_k$. A composition series of $M$ is a filtration 
	$$
	0 = M_n \subset \ldots \subset M_0 = M
	$$
	by sub-Breuil--Kisin modules such that each $M_i/M_{i+1}$ is an irreducible object (i.e. admits no non-zero proper sub-objects $N \in \operatorname{Mod}^{\operatorname{BK}}_k$ such that the cokernel of $N \hookrightarrow M_i/M_{i+1}$ is $k[[u]]$-torsion-free) of $\operatorname{Mod}^{\operatorname{BK}}_k$. Lemma~\ref{ptorsionexactness} implies being irreducible is equivalent to asking that $T(M_i/M_{i+1})$ is an irreducible $G_{K_\infty}$-representation. Lemma~\ref{ptorsionexactness} also implies that composition series for $M$ are in bijection with composition series for $T(M)$.
\end{construction}
\begin{warning}
	It is not the case that the set of irreducible factors of a composition series is independent of the choice of composition series.
\end{warning}
\subsection{Strong divisibility}
In this subsection we define a full-subcategory $\operatorname{Mod}^{\operatorname{SD}}_k \subset \operatorname{Mod}^{\operatorname{BK}}_k$ which we view as an extension of $p$-torsion Fontaine--Laffaille theory to filtrations of length $p$.

\begin{construction}\label{filtrationphiM}
	Let $M$ be an object of $\operatorname{Mod}^{\operatorname{BK}}_k$. Recall $M^\varphi$ is the $k[[u]]$-sub-module of $M[\frac{1}{u}]$ generated by $\varphi(M)$. Equip $M^\varphi$ with a filtration given by $F^i M^\varphi = M^\varphi \cap u^iM$. Let $M^\varphi_k = M^\varphi/u$. We equip this $k$-vector space with the quotient filtration.
\end{construction}

\begin{definition}\label{BKweights}
	If $M \in \operatorname{Mod}_k^{\operatorname{BK}}$ let $\operatorname{Weight}(M)$ be the multiset of integers containing $i$ with multiplicity
	$$
	\operatorname{dim}_k \operatorname{gr}^i(M^\varphi_k)
	$$
\end{definition}
\begin{construction}\label{filtrationMitself}
	Similarly to Construction~\ref{filtrationphiM} we equip $M$ with a filtration by setting $F^iM = \lbrace m \in M \mid \varphi(m) \in u^iM \rbrace$. The semilinear injection 
	$$
	\varphi\colon M \hookrightarrow M^\varphi
	$$ 
	is then a morphism of filtered modules. Let $M_k = M/u$. We equip this $k$-vector space with the quotient filtration.
\end{construction}

\begin{lemma}\label{filtrationMk}
	The injection $\varphi\colon M \hookrightarrow M^\varphi$ induces a functorial $k$-semilinear automorphism of filtered vector spaces
	$$
	M_k \rightarrow M^\varphi_k
	$$
\end{lemma}
\begin{proof}
	All that needs to be checked is that $\varphi\colon M \rightarrow M^\varphi$ induces a $k$-semilinear isomorphism $M_k \rightarrow M_{k}^\varphi$. As $M_k$ and $M^\varphi_k$ have the same dimension over $k$ we only need to check surjectivity. As $M^\varphi$ is the $k[[u]]$-module generated by $\varphi(M) \subset M[\frac{1}{u}]$ surjectivity follows because $\varphi$ is an automorphism on $k = k[[u]]/u$.
\end{proof}

\begin{lemma}\label{equivSD}
	Let $M$ be an object of $\operatorname{Mod}^{\operatorname{BK}}_k$. The following are equivalent:
	\begin{enumerate}
		\item The map $M_k \rightarrow M_k^\varphi$ is an isomorphism of filtered modules.
		\item There exists a $k[[u]]$-basis $(f_i)$ of $M$ and integers $(r_i)$ such that $(u^{r_i}f_i)$ is a $k[[u^p]]$-basis of $\varphi(M)$.
	\end{enumerate}
\end{lemma}
\begin{proof}
	Suppose $M_k \rightarrow M^\varphi_k$ is an isomorphism of filtered modules. We can find integers $r_i$ and elements $g_i \in F^{r_i} M$ whose images in $\operatorname{gr}(M_k)$ form a $k$-basis. As the induced map $\operatorname{gr}(M_k) \rightarrow \operatorname{gr}(M^\varphi_k)$ is an isomorphism it follows that the images of $\varphi(g_i) \in \varphi(M)$ in $\operatorname{gr}(M_k^\varphi)$ form a $k$-basis. Applying Lemma~\ref{liftinggr} with $M = M$, $N = M^\varphi$ and $a \in A$ equal to $u \in k[[u]]$ proves that (1) implies (2) with $f_i = u^{-r_i}\varphi(g_i)$.
	
	To prove (2) implies (1) we use the $f_i$ to give explicit descriptions of the filtration on $M^\varphi_k$. Since $\varphi(M)$ generates $M^\varphi$ over $k[[u]]$ every $m \in M^\varphi$ can be written as $m = \sum \alpha_i (u^{r_i}f_i)$ with $\alpha_i \in k[[u]]$. If $m \in F^jM^\varphi$ then $\alpha_i \in u^{\operatorname{max} \lbrace j-r_i,0 \rbrace}k[[u]] = F^{j-r_i}k[[u]]$ since the $f_i$ form a basis of $M$. Hence 
	$$
	F^j M^\varphi = \sum (F^{j-r_i} k[[u]])(u^{r_i}f_i)
	$$
	and so $F^j M^\varphi_k = \sum_{r_i \geq j} k\overline{f}_i$ where $\overline{f}_i$ denotes the image of $u^{r_i}f_i$ in $M^\varphi_k$. If $g_i \in M$ is such that $\varphi(g_i) = u^{r_i}f_i$ we have $g_i \in F^{j} M$ if $r_i \geq j$. If $\overline{g}_i$ denotes the image of $g_i$ in $M_k$ then since the map $M_k \rightarrow M^\varphi_k$ sends $\overline{g}_i \mapsto \overline{f}_i$, it induces surjections $F^j M_k \rightarrow F^j M^\varphi_k$. Thus $M_k \rightarrow M^\varphi_k$ is an isomorphism in $\operatorname{Fil}(k)$.
\end{proof}
\begin{remark}\label{riremark}
	Note that if we have a basis as in (2) of Lemma~\ref{equivSD} then the above proof shows that $\operatorname{gr}^j(M_k^\varphi) = \sum_{r_i = j} k \overline{f}_i$. Thus the multiset $\lbrace r_i \rbrace$ is equal to $\operatorname{Weight}(M)$.
\end{remark}
\begin{remark}\label{intermsofmatrices}
	Isomorphism classes of objects in $\operatorname{Mod}^{\operatorname{BK}}_k$ can be described explicitly. Choosing a basis and considering the matrix of $\varphi \colon M \hookrightarrow M[\frac{1}{u}]$ with respect to that basis describes a bijection
	\begin{equation}\label{choosebasis}
	\Big\lbrace\text{\parbox{5.5cm}{\centering isomorphism classes of rank $n$  \newline objects of $\operatorname{Mod}^{\operatorname{BK}}_k$}} \Big\rbrace \leftrightarrow \operatorname{GL}_n(k((u))) / \sim 
	\end{equation}
	Here $A \sim B$ if there exists $C \in \operatorname{GL}_n(k[[u]])$ such that $A = C^{-1}B \varphi(C)$. Recall that any invertible matrix over $k((u))$ can be written as $C_1 \Lambda C_2$ where $\Lambda = \operatorname{diag}(u^{r_i})$ and $C_i \in \operatorname{GL}_n(k[[u]])$.
	\begin{itemize}
		\item If $M$ is an object of $\operatorname{Mod}^{\operatorname{BK}}_k$ corresponding under \eqref{choosebasis} to a $\varphi$-conjugacy class represented by $C_1\Lambda C_2$ then the $(r_i) = \operatorname{Weight}(M)$.
		\item  The isomorphism classes of Breuil--Kisin modules satisfying the equivalent conditions of Lemma~\ref{equivSD} identify, via \eqref{choosebasis}, with $\varphi$-conjugacy classes represented by matrices $C_1\Lambda$ with $C_1 \in \operatorname{GL}_n(k[[u]])$ and $\Lambda = \operatorname{diag}(u^{r_i})$.
	\end{itemize}  
\end{remark}

\begin{definition}\label{SDdef}
	Let $\operatorname{Mod}^{\operatorname{SD}}_k \subset \operatorname{Mod}^{\operatorname{BK}}_k$ denote the full subcategory whose objects satisfy the equivalent conditions of Lemma~\ref{equivSD} and have $\operatorname{Weight}(M) \subset [0,p]$. We say such $M$ are strongly divisible.
\end{definition} 
\subsection{Strong divisibility with coefficients}

We reproduce the previous subsection allowing $\mathcal{O}$-coefficients.
\begin{definition}
	Let $\operatorname{Mod}^{\operatorname{BK}}_k(\mathcal{O})$ denote the full subcategory of $\operatorname{Mod}^{\operatorname{BK}}_K(\mathcal{O})$ whose objects are finite free over $k[[u]] \otimes_{\mathbb{F}_p} \mathbb{F}$. This is equivalent to being free over $k[[u]]$ and killed by $\varpi$ after Corollary~\ref{freethings}.
\end{definition}
\begin{remark}
	As in Construction~\ref{O-semilinear} each $M \in \operatorname{Mod}^{\operatorname{BK}}_k(\mathcal{O})$ decomposes as
	$$
	M = \prod_{\tau \in \operatorname{Hom}_{\mathbb{F}_p}(k,\mathbb{F})} M_{\tau}
	$$
	with each $M_{\tau}$ a finite free module over $\mathbb{F}[[u]]$. Since the filtration on $M$ is by $k[[u]] \otimes_{\mathbb{F}_p} \mathbb{F}$-sub-modules this is a decomposition of filtered modules. Thus $M_k = \prod_{\tau} M_{k,\tau}$ as filtered modules (each $M_{k,\tau}$ being a filtered $\mathbb{F}$-vector space). Analogous statements hold for $M^\varphi$ and $M^\varphi_k$.
\end{remark}

\begin{definition}
	For $\tau \in \operatorname{Hom}_{\mathbb{F}_p}(k,\mathbb{F})$ let $\operatorname{Weight}_\tau(M)$ be the multiset of integers which contains $i$ with multiplicity equal to 
	$$
	\operatorname{dim}_{\mathbb{F}} \operatorname{gr}^i(M^\varphi_{k,\tau})
	$$
	Since $M^\varphi_k = \prod M^\varphi_{k,\tau}$ we have that $\operatorname{Weight}(M)$ equals the union over all $\tau$ of $[\mathbb{F}:k]$ copies of  $\operatorname{Weight}_\tau(M)$.
\end{definition}

The following is a version of Lemma~\ref{equivSD} for objects of $\operatorname{Mod}^{\operatorname{BK}}_k(\mathcal{O})$ and is proved in exactly the same fashion.
\begin{lemma}\label{equivSDO}
	Let $M$ be an object of $\operatorname{Mod}^{\operatorname{BK}}_k(\mathcal{O})$. Then the following are equivalent:
	\begin{enumerate}
		\item The semilinear map $M_k \rightarrow M^\varphi_k$ is an isomorphism of filtered modules.
		\item For $\tau \in \operatorname{Hom}_{\mathbb{F}_p}(k,\mathbb{F})$ there exists an $\mathbb{F}[[u]]$-basis $(f_i)$ of $M_\tau$ and integers $(r_i)$ such that $(u^{r_i}f_i)$ is an $\mathbb{F}[[u^p]]$-basis of $\varphi(M)_\tau$.
	\end{enumerate}
\end{lemma}
\begin{remark}\label{WtremarkO}
	As in Remark~\ref{riremark} if bases as in (2) of Lemma~\ref{equivSDO} exist then the multiset $\lbrace r_{i,\tau} \rbrace$ equals $\operatorname{Weight}_\tau(M)$.
\end{remark}
\begin{remark}
	There is the following analogue of Remark~\ref{intermsofmatrices} for $\operatorname{Mod}^{\operatorname{BK}}_k(\mathcal{O})$. Choosing $\mathbb{F}[[u]]$-bases for each $M_{\tau}$ and taking the matrices representing $\varphi$ with respect to these bases describes a bijection
	$$
	\Big\lbrace\text{\parbox{5.5cm}{\centering isomorphism classes of rank $n$  \newline objects of $\operatorname{Mod}^{\operatorname{BK}}_k(\mathcal{O})$}} \Big\rbrace \leftrightarrow \operatorname{GL}_n(\mathbb{F}((u)))^f / \sim 
	$$
	where $f = [K:\mathbb{Q}_p]$ and where two $f$-tuples of matrices satisfy $(A_\tau) \sim (B_\tau)$ if there exist $C_{\tau} \in \operatorname{GL}_n(\mathbb{F}[[u]])$ such that $A_\tau = C_\tau^{-1}B_{\tau}\varphi(C_{\tau \circ \varphi})$ for all $\tau$. Each $A_{\tau}$ can be written as $C_{\tau}\Lambda_\tau C'_\tau$ with $C_\tau,C_{\tau}' \in \operatorname{GL}_n(\mathbb{F}[[u]])$ and $\Lambda_\tau = \operatorname{diag}(u^{r_{i,\tau}})$.
	\begin{itemize}
		\item The multiset $\lbrace r_{i,\tau} \rbrace$ is the multiset $\operatorname{Weight}_\tau(M)$.
		\item The $M$ which satisfy Lemma~\ref{equivSDO} correspond to classes represented by an $f$-tuple of matrices $(A_\tau)$ such that each $A_{\tau} = C_{\tau} \Lambda_\tau$.
	\end{itemize}
\end{remark}

\begin{definition}
	Let $\operatorname{Mod}^{\operatorname{SD}}_k(\mathcal{O}) \subset \operatorname{Mod}^{\operatorname{BK}}_k(\mathcal{O})$ denote the full subcategory whose objects are strongly divisible when viewed as objects of $\operatorname{Mod}^{\operatorname{BK}}_k$.
\end{definition}

\subsection{Subquotients} We now show $\operatorname{Mod}^{\operatorname{SD}}_k$ and $\operatorname{Mod}^{\operatorname{SD}}_k(\mathcal{O})$ are closed under subquotients.
\begin{remark}\label{exactseqoffilteredstuff}
	If $M \in \operatorname{Mod}^{\operatorname{BK}}_k$ then there are exact sequences
	$$
	\begin{aligned}
	0 \rightarrow \operatorname{gr}^{i-1}(M^\varphi) \xrightarrow{u}& \operatorname{gr}^i(M^\varphi) \rightarrow \operatorname{gr}^i(M^\varphi_k) \rightarrow 0 \\
	0 \rightarrow \operatorname{gr}^{i-p}(M) \xrightarrow{u}& \operatorname{gr}^i(M) \rightarrow \operatorname{gr}^i(M_k) \rightarrow 0
	\end{aligned}
	$$
	The first is just the exact sequence \eqref{exactsequenceusedlater} in the case $M = M$ and $N = M^\varphi$ with $A = k[[u]]$ and $a = u$. The second exact sequence is obtained similarly (using that $F^i(uM) = u(F^{i - p} M)$).
\end{remark}

\begin{lemma}\label{N->Pstrict}
	Let $0 \rightarrow M \rightarrow N \rightarrow P \rightarrow 0$ be an exact sequence in $\operatorname{Mod}^{\operatorname{BK}}_k$.
	\begin{enumerate}
		\item The map $N \rightarrow P$ is strict when viewed as a map of filtered modules if and only if $0 \rightarrow M_k \rightarrow N_k \rightarrow P_k \rightarrow 0$ is an exact sequence in $\operatorname{Fil}(k)$ in the sense of Notation~\ref{filteredexact}.
		\item The map $N^\varphi \rightarrow P^\varphi$ is strict if and only if $0 \rightarrow M_k^\varphi \rightarrow N_k^\varphi \rightarrow P_k^\varphi \rightarrow 0$ is exact in $\operatorname{Fil}(k)$
		\item Statement (2) is equivalent to $M_k^\varphi \rightarrow N_k^\varphi$ being strict, which is equivalent to $N_k^\varphi \rightarrow P_k^\varphi$ being strict.
	\end{enumerate}
\end{lemma}
\begin{proof}
	Note that $M \rightarrow N$ is strict as a map of filtered modules. To see this suppose $m \in M \cap F^i N$, then $\varphi(m) \in \varphi(M) \cap u^iN \subset M[\frac{1}{u}] \cap u^iN$. Since $M \rightarrow N$ has $u$-torsion-free cokernel $M[\frac{1}{u}] \cap u^iN = u^iM$. Thus $m \in F^iM$. Similarly $M^\varphi \rightarrow N^\varphi$ is strict. Hence $N \rightarrow P$ is strict if and only if $0 \rightarrow \operatorname{gr}^i(M) \rightarrow \operatorname{gr}^i(N) \rightarrow \operatorname{gr}^i(P) \rightarrow 0$ is exact for each $i$ and likewise $N^\varphi \rightarrow P^\varphi$ is strict if and only if $0 \rightarrow \operatorname{gr}^i(M^\varphi) \rightarrow \operatorname{gr}^i(N^\varphi) \rightarrow \operatorname{gr}^i(P^\varphi) \rightarrow 0$ is exact (Lemma~\ref{fantasticstuff}).
	
	Using the second exact sequence of Remark~\ref{exactseqoffilteredstuff} we obtain the following commutative diagram with exact rows.
	$$
	\begin{tikzcd}[column sep =small, row sep=small]
	~& 0 \arrow{d} & 0 \arrow{d} & 0 \arrow{d} & \\
	0 \arrow{r} & \operatorname{gr}^{i-p}(M) \arrow{r}{u} \arrow{d} & \operatorname{gr}^{i}(M) \arrow{r} \arrow{d} & \operatorname{gr}^{i}(M_k) \arrow{r} \arrow{d} & 0 \\
	0 \arrow{r} & \operatorname{gr}^{i-p}(N) \arrow{r}{u} \arrow{d} & \operatorname{gr}^{i}(N) \arrow{r} \arrow{d} & \operatorname{gr}^{i}(N_k) \arrow{r} \arrow{d} & 0 \\
	0 \arrow{r} & \operatorname{gr}^{i-p}(P) \arrow{r}{u} \arrow{d} & \operatorname{gr}^{i}(P) \arrow{r} \arrow{d} & \operatorname{gr}^{i}(P_k) \arrow{r} \arrow{d} & 0 \\
	~ & 0 & 0 & 0 &
	\end{tikzcd}
	$$
	The previous paragraph shows that if $N \rightarrow P$ is strict then the left and middle columns are exact, and so the right column is exact also. Conversely if the right column is exact then one proves the middle column is exact by increasing induction on $i$ (for small enough $i$ the left column will be zero). This proves (1). The same argument but with the diagram replaced with the diagram obtained by considering the first exact sequence of Remark~\ref{exactseqoffilteredstuff} proves (2) also.
	
	It remains to show that if $M_k^\varphi \rightarrow N_k^\varphi$ or $N_k^\varphi \rightarrow P_k^\varphi$ is strict then $0 \rightarrow M_k^\varphi \rightarrow N_k^\varphi \rightarrow P_k^\varphi \rightarrow 0$ is exact. It suffices to show that $\sum_{i \in \operatorname{Weight}(M)} i + \sum_{i \in \operatorname{Weight}(P)} i = \sum_{i \in \operatorname{Weight}(N)} i$ after Corollary~\ref{exactsequencesandsums}. Remark~\ref{intermsofmatrices} says that $\sum_{i \in \operatorname{Weight}(M)} i$ equals the $u$-adic valuation of the determinant of $\varphi \colon M \rightarrow M[\frac{1}{u}]$ (in any choice of basis). Since this is clearly additive on exact sequences the lemma follows.
\end{proof}
\begin{lemma}\label{extSD}
	Let $0 \rightarrow M \rightarrow N \rightarrow P \rightarrow 0$ be an exact sequence in $\operatorname{Mod}^{\operatorname{BK}}_k$. Suppose $M$ and $P$ satisfy the equivalent conditions of Lemma~\ref{equivSD}.  If $N \rightarrow P$ is strict then $N$ satisfies the equivalent conditions of Lemma~\ref{equivSD} also.
\end{lemma}
\begin{proof}
	Consider the following commutative diagram.
	$$
	\begin{tikzcd}[column sep =small, row sep=small]
	0 \arrow{r} & \operatorname{gr}^i(M_k^\varphi) \arrow{r} & \operatorname{gr}^i(N_k^\varphi) \arrow{r} & \operatorname{gr}^i(P_k^\varphi) \arrow{r} & 0 \\
	0 \arrow{r} & \operatorname{gr}^i(M_k) \arrow{r} \arrow{u}{\wr} & \operatorname{gr}^i(N_k) \arrow{r} \arrow{u} & \operatorname{gr}^i(P_k) \arrow{r} \arrow{u}{\wr} & 0
	\end{tikzcd}
	$$
	The left and right vertical arrows are isomorphisms by assumption. Since $N \rightarrow P$ is strict, part (1) of Lemma~\ref{N->Pstrict} implies the bottom row is exact. Thus $\operatorname{gr}^i(N_k^\varphi) \rightarrow \operatorname{gr}^i(P_k^\varphi)$ is surjective and so $N_k^\varphi \rightarrow P_k^\varphi$ is strict by Lemma~\ref{fantasticstuff}. Part (3) of Lemma~\ref{N->Pstrict} then implies the top row is exact. We conclude that $N_k \rightarrow N_k^\varphi$ is an isomorphism in $\operatorname{Fil}(k)$.
\end{proof}

\begin{lemma}\label{subquotisSD}
	Let $0 \rightarrow M \rightarrow N \rightarrow P \rightarrow 0$ be an exact sequence in $\operatorname{Mod}^{\operatorname{BK}}_k$. Suppose that $N$ satisfies the equivalent conditions of Lemma~\ref{equivSD} and that $M_k \rightarrow N_k$ is strict. Then $N \rightarrow P$ is strict and $M$ and $P$ also satisfy the equivalent conditions of Lemma~\ref{equivSD}.
\end{lemma}
\begin{proof}
	The following diagram of objects in $\operatorname{Fil}(k)$ commutes.
	$$
	\begin{tikzcd}[column sep =small, row sep=small]
	M^\varphi_k \arrow{r}&  N^\varphi_k \\
	M_k \arrow{u} \arrow{r} & N_k \arrow{u}
	\end{tikzcd}
	$$
	As maps of $k$-vector spaces the horizontal arrows are injective and the vertical arrows are isomorphisms. By assumption the maps $M_k \rightarrow N_k$ and $N_k \rightarrow N^\varphi_k$ are strict. It follows that $M^\varphi_k \rightarrow N^\varphi_k$ and $M_k \rightarrow M^\varphi_k$ are strict also.
	
	The following is also a commutative diagram in $\operatorname{Fil}(k)$.
	$$
	\begin{tikzcd}[column sep =small, row sep=small]
	N^\varphi_k \arrow{r} & P^\varphi_k \\
	N_k \arrow{u} \arrow{r} & P_k \arrow{u}
	\end{tikzcd}
	$$
	As maps of $k$-vector spaces the vertical maps are isomorphisms and the horizontal arrows are surjections. By assumption the leftmost vertical arrow is strict. Using part (3) of Lemma~\ref{N->Pstrict}, $M_k^\varphi \rightarrow N^\varphi_k$ being strict implies $N^\varphi_k \rightarrow P^\varphi_k$ is strict. It follows that $P_k \rightarrow P^\varphi_k$ and $N_k \rightarrow P_k$ are strict. Thus $M$ and $P$ are as in Lemma~\ref{equivSD} and after (1) of Lemma~\ref{N->Pstrict} we know $N \rightarrow P$ is strict.
\end{proof}
\begin{lemma}\label{propertyP}
	Suppose $N$ is strongly divisible. If $0 \rightarrow M \rightarrow N \rightarrow P \rightarrow 0$ is an exact sequence in $\operatorname{Mod}^{\operatorname{BK}}_k$ then $M_k \rightarrow N_k$ is strict.
\end{lemma} 
\begin{proof}
	We have a commutative diagram with exact rows (Remark~\ref{exactseqoffilteredstuff})
	$$
	\begin{tikzcd}[column sep =small, row sep=small]
	0 \arrow{r} & \operatorname{gr}^{i-p}(M) \arrow{r} \arrow{d} & \operatorname{gr}^{i}(M) \arrow{r} \arrow{d} & \operatorname{gr}^i(M_k) \arrow{r} \arrow{d}{\alpha} & 0 \\
	0 \arrow{r} & \operatorname{gr}^{i-p}(N) \arrow{r} & \operatorname{gr}^{i}(N) \arrow{r} & \operatorname{gr}^i(N_k) \arrow{r} & 0
	\end{tikzcd}
	$$
	One knows that $M \rightarrow N$ is strict (as was shown in the first paragraph of the proof of Lemma~\ref{N->Pstrict}) so the left and middle vertical arrows are injective by Lemma~\ref{fantasticstuff}. We have to show $\alpha$ is injective for every $i$.
	
	For injectivity of $\alpha$ when $i< p$ we argue as follows. As $\operatorname{Weight}(N) \subset [0, p]$, and because $N_k \cong N^\varphi_k$, we have $\operatorname{gr}^i(N_k) = 0$ for $i<0$. Hence $\operatorname{gr}^i(N) = \operatorname{gr}^{i-p}(N)$ for $i<0$. This implies $\operatorname{gr}^i(N) = 0$ for $i<0$ because for small enough $i$, $F^i N = N$. Using the diagram we deduce that $\operatorname{gr}^i(M) = 0$ for $i<0$ also, and that for $i<p$ we have $\operatorname{gr}^i(M) = \operatorname{gr}^i(M_k)$ and $\operatorname{gr}^i(N) = \operatorname{gr}^i(N_k)$. This proves $\alpha$ is injective when $i<p$. 
	
	For injectivity of $\alpha$ when $i \geq  p$ it suffices to show $F^i N_k = 0$ for $i> p$ (because then $F^iM_k = 0$ for $i>p$ so $\alpha$ is just the zero map when $i>p$ and when $i=p$, $\alpha$ is the inclusion $F^iM_k \rightarrow F^iN_k$). Let us prove this is the case. Since $\operatorname{Weight}(N) \subset [0, p]$ we have $\operatorname{gr}^i(N_k) = 0$ for $i> p$; it suffices to show $F^i N_k =0$ for $i>>p$. But $N_k$ is both Hausdorff (being a quotient of $N$, which is Hausdorff) and a finite dimensional $k$-vector space, this forces $F^i N_k$ to vanish for large $i$. So we are done.
\end{proof}

Putting all this together we deduce the following.
\begin{proposition}\label{SDsubquot}
	Let $0 \rightarrow M \rightarrow N \rightarrow P \rightarrow 0$ be an exact sequence in $\operatorname{Mod}^{\operatorname{BK}}_k$. 
	\begin{enumerate}
		\item If $N \in \operatorname{Mod}^{\operatorname{SD}}_k$ then $M$ and $P$ are strongly divisible and the sequence 
		$$
		0 \rightarrow M^\varphi_k \rightarrow N^\varphi_k \rightarrow P^\varphi_k \rightarrow 0
		$$
		is exact in $\operatorname{Fil}(k)$. Thus $\operatorname{Weight}(N) = \operatorname{Weight}(M) \cup \operatorname{Weight}(P)$. 
		\item If $P, M \in \operatorname{Mod}^{\operatorname{SD}}_k$ then $N \in \operatorname{Mod}^{\operatorname{SD}}_k$ if and only if $N \rightarrow P$ is strict.
	\end{enumerate}
	\begin{proof}
		(1) follows from Lemma~\ref{N->Pstrict}, Lemma~\ref{subquotisSD} and Lemma~\ref{propertyP}. For (2) use Lemma~\ref{extSD}.
	\end{proof}
\end{proposition}
\begin{proposition}\label{SDsubquotO}
	Let $0 \rightarrow M \rightarrow N \rightarrow P \rightarrow 0$ be an exact sequence in $\operatorname{Mod}^{\operatorname{BK}}_k(\mathcal{O})$. 
	\begin{enumerate}
		\item If $N \in \operatorname{Mod}^{\operatorname{SD}}_k(\mathcal{O})$ then $M$ and $P$ are both strongly divisible and for each $\tau \in \operatorname{Hom}_{\mathbb{F}_p}(k,\mathbb{F})$ we have $\operatorname{Weight}_\tau(N) = 
		\operatorname{Weight}_\tau(M) \cup \operatorname{Weight}_\tau(P)$.
		\item If $M,P \in \operatorname{Mod}^{\operatorname{SD}}_k(\mathcal{O})$ then $N \in \operatorname{Mod}^{\operatorname{SD}}_k(\mathcal{O})$ if and only if $N \rightarrow P$ is strict.
	\end{enumerate}
\end{proposition}
\begin{proof}
	This is immediate from Proposition~\ref{SDsubquot}. In particular we point out that the exact sequence in (1) of Proposition~\ref{SDsubquot} is functorial and so is an exact sequence of $k \otimes_{\mathbb{F}_p} \mathbb{F}$-modules. Thus it decomposes into exact sequences
	$$
	0 \rightarrow M_{k,\tau}^\varphi \rightarrow N_{k,\tau}^\varphi \rightarrow P_{k,\tau}^\varphi \rightarrow 0
	$$ 
	which shows $\operatorname{Weight}_\tau(N) = \operatorname{Weight}_\tau(M) \cup \operatorname{Weight}_\tau(P)$.
\end{proof}
\section{Irreducible objects}\label{irreducible}

Provided $\mathbb{F}$ is sufficiently large, irreducible $\mathbb{F}$-representations of $G_K$ and $G_{K_\infty}$ are induced from characters (Lemma~\ref{irred}). In this section and the next we investigate the extent with which this is true for objects of $\operatorname{Mod}^{\operatorname{SD}}_k(\mathcal{O})$. Throughout assume $k \subset \mathbb{F}$.
\subsection{Rank ones}
Recall from Construction~\ref{O-semilinear} how $\mathfrak{S} \otimes_{\mathbb{Z}_p} \mathcal{O}$ is made into an $\mathcal{O}[[u]]$-algebra. Then $k[[u]] \otimes_{\mathbb{F}_p} \mathbb{F}$ becomes an $\mathbb{F}[[u]]$-algebra. Also let $e_\tau \in k[[u]] \otimes_{\mathbb{F}_p} \mathbb{F}$ denote the image of the idempotent $\widetilde{e}_\tau \in \mathfrak{S} \otimes_{\mathbb{Z}_p} \mathcal{O}$ defined in Construction~\ref{O-semilinear}. Thus $\varphi(e_{\tau \circ \varphi}) = e_\tau$.

The next lemma is proven by an easy change of basis argument (see \cite[Lemma 6.2]{GLS})
\begin{lemma}\label{rankone}
	Fix $\tau_0 \in \operatorname{Hom}_{\mathbb{F}_p}(k,\mathbb{F})$. Let $M \in \operatorname{Mod}^{\operatorname{BK}}_k(\mathcal{O})$ be of rank one over $k[[u]] \otimes_{\mathbb{F}_p} \mathbb{F}$. Then $M$ is isomorphic to a Breuil--Kisin module
	$$
	N = k[[u]] \otimes_{\mathbb{F}_p} \mathbb{F}, \qquad \varphi_N(1) = (x)\sum u^{r_\tau} e_\tau
	$$
	where $r_\tau \in \mathbb{Z}$ and where $(x) = xe_{\tau_0} + \sum_{\tau \neq \tau_0} e_\tau$ for some $x \in \mathbb{F}^\times$.
\end{lemma}
\begin{remark}
	If $N$ is as in Lemma~\ref{rankone} then $\operatorname{Weight}_\tau(N) = \lbrace r_\tau \rbrace$. Note also that $N$ satisfies the equivalent conditions of Lemma~\ref{equivSDO}. Thus $N \in \operatorname{Mod}^{\operatorname{SD}}_k(\mathcal{O})$ if and only if $r_{\tau} \in [0,p]$.
\end{remark}
\begin{proposition}\label{rankonesapplyT}
	If $N$ is as in Lemma~\ref{rankone} then the $G_{K_\infty}$-action on $T(N)$ is through the restriction to $G_{K_\infty}$ of the character
	$$
	\psi_x\prod_{\tau} \omega_\tau^{-r_\tau}
	$$
	Here $\psi_x$ denotes the unramified character sending the geometric Frobenius to $x$, and the $\omega_\tau$ are the characters defined in the paragraph after the proof of Lemma~\ref{uniquertheta}.
\end{proposition}
\begin{proof}
	This is \cite[Proposition 6.7]{GLS}. However note that in \emph{loc. cit.} they contravariantly associate a $G_{K_\infty}$-representation to Breuil--Kisin module; this is why the character appearing here is the inverse of that in \emph{loc. cit}.
\end{proof}
\subsection{Induction and restriction}
\begin{notation}\label{restrictionUNRAM}
	Let $L/K$ be the unramified extension corresponding to a finite extension $l/k$, and let $L_\infty = K_\infty L$. Set $\mathfrak{S}_L = W(l)[[u]]$. Extension of scalars along the inclusion $f\colon \mathfrak{S} \rightarrow \mathfrak{S}_L$ describes a functor
	$$
	f^*\colon \operatorname{Mod}^{\operatorname{BK}}_K \rightarrow \operatorname{Mod}^{\operatorname{BK}}_L
	$$
	For $M \in \operatorname{Mod}^{\operatorname{BK}}_K$ the module $f^*M = M \otimes_{\mathfrak{S}} \mathfrak{S}_L$ is made into a Breuil--Kisin module via the semilinear map $m \otimes s \mapsto \varphi_M(m)\otimes \varphi(s)$; this map induces the isomorphism
	$$
	(\varphi^*f^*M)[\tfrac{1}{E}] = (f^*\varphi^*M)[\tfrac{1}{E}] = f^*(\varphi^*M[\tfrac{1}{E}]) \xrightarrow{f^*\varphi_M} f^*(M[\tfrac{1}{E}]) = (f^*M)[\tfrac{1}{E}]
	$$
	where the first $=$ comes from the fact that $\varphi \circ f = f \circ \varphi$. The natural isomorphism 
	$$
	f^*M \otimes_{\mathfrak{S}_L} W(C^\flat) \cong M \otimes_{\mathfrak{S}} W(C^\flat)
	$$
	is clearly $\varphi,G_{L_\infty}$-equivariant so $T(f^*M) =  T(M)|_{G_{L_\infty}}$.
\end{notation}

\begin{notation}\label{inductionUNRAM}
	With notation as in Notation~\ref{restrictionUNRAM}, restriction of scalars along $f$ induces a functor
	$$
	f_*\colon \operatorname{Mod}^{\operatorname{BK}}_L \rightarrow \operatorname{Mod}^{\operatorname{BK}}_K
	$$
	If $M \in \operatorname{Mod}^{\operatorname{BK}}_L$ we equip $f_*M$ with the obvious semilinear map $m \mapsto \varphi_M(m)$. Let us verify that this makes $f_*M$ into a Breuil--Kisin module. The semilinear map induces the composite:
	$$
	(\varphi^*f_*M)[\tfrac{1}{E}] \rightarrow (f_*\varphi^*M)[\tfrac{1}{E}] = f_*(\varphi^*M [\tfrac{1}{E}]) \xrightarrow{f_*\varphi_M} f_*(M[\tfrac{1}{E}]) = (f_*M)[\tfrac{1}{E}]
	$$
	which we claim is an isomorphism. It suffices to check the natural map $\varphi^*f_*M \rightarrow f_*\varphi^*M$ is an isomorphism, and this follows because the commutative diagram
	$$
	\begin{tikzcd}[column sep =small, row sep=small]
	\mathfrak{S} \arrow{r}{f} & \mathfrak{S}_L \\
	\mathfrak{S} \arrow{r}{f} \arrow{u}{\varphi} & \mathfrak{S}_L \arrow{u}[swap]{\varphi}
	\end{tikzcd}
	$$
	is a pushout. 
\end{notation}

\begin{lemma}\label{frobrecBK}
	For all $M \in \operatorname{Mod}^{\operatorname{BK}}_K$ and $N \in \operatorname{Mod}^{\operatorname{BK}}_L$ there are functorial isomorphisms
	$$
	\operatorname{Hom}(M,f_*N) \cong f_*\operatorname{Hom}(f^*M,N)
	$$
	in $\operatorname{Mod}^{\operatorname{BK}}_K$.
\end{lemma}
\begin{proof}
	The standard adjunction between $f^*$ and $f_*$ provides functorial $\mathfrak{S}$-linear isomorphisms $\operatorname{Hom}_{\mathfrak{S}}(M,f_*N) \rightarrow \operatorname{Hom}_{\mathfrak{S}_L}(f^*M,N)$. Explicitly this map sends $\alpha$ onto the homomorphism $m \otimes s \mapsto s \alpha(m)$. As this is $\varphi$-equivariant we get isomorphisms as claimed.
\end{proof}

\begin{lemma}\label{yoneda}
	Let $N \in \operatorname{Mod}^{\operatorname{BK}}_L$. Then there are functorial identifications $\iota_N:T(f_*N) \rightarrow \operatorname{Ind}_{L_\infty}^{K_\infty} T(N)$ such that the diagram
	$$
	\begin{tikzcd}[column sep =small, row sep=small]
	\operatorname{Hom}_{\operatorname{BK}}(M,f_*N) \arrow{r}{\ref{frobrecBK}} \arrow{d}{g \mapsto \iota_N \circ T(g)} & \operatorname{Hom}_{\operatorname{BK}}(f^*M,N) \arrow{d}{T}\\
	\operatorname{Hom}_{G_{K_\infty}}(T(M),\operatorname{Ind}_{L_\infty}^{K_\infty} T(N)) \arrow{r}{(\operatorname{Frob})} & \operatorname{Hom}_{G_{L_\infty}}(T(M)|_{G_{L_\infty}},T(N))
	\end{tikzcd}
	$$
	commutes for all $M \in \operatorname{Mod}^{\operatorname{BK}}_K$. The top horizontal arrow is obtained from the identification in Lemma~\ref{frobrecBK} by taking $\varphi$-invariants, and the lower horizontal arrow is given by Frobenius reciprocity.
\end{lemma}
\begin{proof}
	Let $\mathcal{O}_{\mathcal{E},L}$ be the $p$-adic completion of $\mathfrak{S}_L[\frac{1}{u}]$. The map $f: \mathfrak{S} \rightarrow \mathfrak{S}_L$ extends to a map $f: \mathcal{O}_{\mathcal{E}} \rightarrow \mathcal{O}_{\mathcal{E},L}$ and so we can make sense of the operations $f^*$ and $f_*$ on etale $\varphi$-modules. Write $M^{\operatorname{et}} = M \otimes_{\mathfrak{S}} \mathcal{O}_{\mathcal{E}}$ and $N^{\operatorname{et}} = N \otimes_{\mathfrak{S}_L} \mathcal{O}_{\mathcal{E},L}$. Then clearly $f^* (M^{\operatorname{et}}) = (f^*M)^{\operatorname{et}}$, and because $\mathcal{O}_{\mathcal{E},L} = \mathcal{O}_{\mathcal{E}} \otimes_{\mathfrak{S}} \mathfrak{S}_L$ we also have that $f_*(N^{\operatorname{et}}) = (f_*N)^{\operatorname{et}}$. We obtain maps
	$$
	\operatorname{Hom}_{\operatorname{BK}}(M,f_*N) \rightarrow \operatorname{Hom}_{\operatorname{et}}(M^{\operatorname{et}},f_*N^{\operatorname{et}}), \qquad \operatorname{Hom}_{\operatorname{BK}}(f^*M,N) \rightarrow \operatorname{Hom}_{\operatorname{et}}(f^*M^{\operatorname{et}},N^{\operatorname{et}})
	$$
	which commute with $T$. The analogue of Lemma~\ref{frobrecBK} in the setting of etale $\varphi$-modules is proved in exactly the same way, and the obtained identification is compatible with the maps above. Thus to prove the lemma we may replace $\operatorname{Hom}_{\operatorname{BK}}$ with $\operatorname{Hom}_{\operatorname{et}}$ (homsets in the category of etale $\varphi$-modules) and $M$ and $N$ with $M^{\operatorname{et}}$ and $N^{\operatorname{et}}$ in the diagram of the lemma.
	
	Since $M^{\operatorname{et}} \mapsto T(M^{\operatorname{et}})$ is an equivalence of categories, the map $
	(\operatorname{Frob}) \circ T \circ \eqref{frobrecBK} \circ T^{-1}$ describes an identification 
	\begin{equation}\label{nathom}
	\operatorname{Hom}_{G_{K_\infty}}(V,T(f_*N)) \rightarrow \operatorname{Hom}_{G_{K_\infty}}(V,\operatorname{Ind}_{L_\infty}^{K_\infty} T(N))
	\end{equation}
	for any continuous $G_{K_\infty}$-representation $V$ on a finitely generated $\mathbb{Z}_p$-module. As \eqref{nathom} is functorial in $V$ Yoneda's lemma provides the isomorphism $\iota_N$. As \eqref{nathom} is functorial in $N$ we see that $\iota_N$ is functorial.
\end{proof}

\begin{lemma}\label{SDinductionrestriction}
	Assume $k \subset l \subset \mathbb{F}$.
	\begin{enumerate}
		\item If $M \in \operatorname{Mod}^{\operatorname{SD}}_k(\mathcal{O})$ then $f^*M \in \operatorname{Mod}^{\operatorname{SD}}_l(\mathcal{O})$ and for each $\theta \in \operatorname{Hom}_{\mathbb{F}_p}(l,\mathbb{F})$ we have 
		$$
		\operatorname{Weight}_\theta(f^*M) = \operatorname{Weight}_{\theta|_k}(M)
		$$
		\item If $N \in \operatorname{Mod}^{\operatorname{SD}}_l(\mathcal{O})$ then $f_*N \in \operatorname{Mod}^{\operatorname{SD}}_k(\mathcal{O})$ and 
		$$
		\operatorname{Weight}_\tau(f_*N) = \bigcup_{\theta|_k = \tau} \operatorname{Weight}_\theta(N)
		$$
	\end{enumerate} 
\end{lemma}
\begin{proof}
	By functoriality both $f^*$ and $f_*$ preserve $\mathcal{O}$-actions. Note that the inclusion $k[[u]] \otimes_{\mathbb{F}_p} \mathbb{F} \rightarrow l[[u]] \otimes_{\mathbb{F}_p} \mathbb{F}$ sends  the idempotents $e_\tau \mapsto \sum_{\theta|_k = \tau} e_\theta$. Thus $(f^*M)_\theta =  M_{\theta|_k}$ and $(f_*N)_\tau = \prod_{\theta|_k = \tau} N_\theta$. Both (1) and (2) then follow by verifying the second condition of Lemma~\ref{equivSDO}.
\end{proof}

\subsection{Approximation by induced Breuil--Kisin modules}

We consider the situation given in Notation~\ref{restrictionUNRAM}. Thus $L/K$ is a finite unramified extension, corresponding to an extension $l/k$ of residue fields, and $L_\infty = L(\pi^{1/p^\infty})$. We also have the map $f\colon \mathfrak{S} \rightarrow \mathfrak{S}_L$. 
\begin{lemma}\label{intermsofN}
	Suppose $M \in \operatorname{Mod}^{\operatorname{SD}}_k(\mathcal{O})$ and assume that $T(M) \cong \operatorname{Ind}_{L_\infty}^{K_\infty} T'$. Then there exists an $N \in \operatorname{Mod}^{\operatorname{SD}}_k(\mathcal{O})$ with $T(N) = T'$, together with a $\varphi$-equivariant inclusion 
	$$
	M \hookrightarrow f_*N
	$$
	of $k[[u]] \otimes_{\mathbb{F}_p} \mathbb{F}$-modules which becomes an isomorphism after inverting $p$.
\end{lemma}

\begin{proof}
	There is a non-zero map $T(M)|_{G_{L_\infty}} \rightarrow T'$ corresponding under Frobenius reciprocity to the isomorphism $T(M) \cong \operatorname{Ind}^{K_\infty}_{L_\infty} T'$. Thus there is a surjection $f^*M \rightarrow N$ where $N \in \operatorname{Mod}^{\operatorname{BK}}_l(\mathcal{O})$ is of rank one with $T(N) = T'$ (Lemma~\ref{ptorsionexactness}). Applying Lemma~\ref{yoneda} to $f^*M \rightarrow N$ we obtain a map 
	\begin{equation*}\label{inclusion}
	M \rightarrow f_*N
	\end{equation*} 
	which, after applying $T$, induces the identification $T(N) = T'$.  Thus $M \rightarrow f_*N$ becomes an isomorphism after inverting $u$ and is, in particular, injective. Lemma~\ref{SDinductionrestriction} implies $f^*M \in \operatorname{Mod}^{\operatorname{SD}}_l(\mathcal{O})$, since $M \in \operatorname{Mod}^{\operatorname{SD}}_k(\mathcal{O})$. Therefore $N \in \operatorname{Mod}^{\operatorname{SD}}_k(\mathcal{O})$ by Proposition~\ref{SDsubquotO}. 
	\end{proof}
	
	When $T(M)$ is irreducible and $\mathbb{F}$ is sufficiently large $T(M)$ is induced from a character, and so Lemma~\ref{intermsofN} produces an inclusion $M \hookrightarrow f_*N$ with $N$ of rank one. Lemma~\ref{rankone} allows us to describe $N$ explicitly. In this case we would like to know which submodules of $f_*N$ arise in this way. The following example shows that there are non-trivial possibilities.
	
	\subsection{An example}\label{example}   Take $K = \mathbb{Q}_p$ and let $L/K$ be of degree $5$ with residue extension $l/k$. Let $N \in \operatorname{Mod}^{\operatorname{SD}}_l(\mathcal{O})$ be the rank one object defined by
	$$
	N = l[[u]] \otimes_{\mathbb{F}_p} \mathbb{F}, \qquad \varphi_N(1) = u^x e_{\theta \circ \varphi^4} + u^n e_{\theta \circ \varphi^3} + e_{\theta \circ \varphi^2} + u^n e_{\theta \circ \varphi} + e_{\theta}
	$$
	Here we have fixed $\theta \in \operatorname{Hom}_{\mathbb{F}_p}(l,\mathbb{F})$ and $1 \leq n \leq p, 0 \leq x \leq p$. Let $M \subset f_*N$ be the sub-module generated over $\mathbb{F}[[u]]$ by $e_{\theta \circ \varphi^4},e_{\theta \circ \varphi^3} + e_{\theta \circ \varphi}, e_{\theta \circ \varphi^2}, ue_{\theta \circ \varphi},e_{\theta}$. One computes that
	$$
	\varphi(e_{\theta \circ \varphi^4}, e_{\theta \circ \varphi^3} + e_{\theta \circ \varphi}, e_{\theta \circ \varphi^2}, ue_{\theta \circ \varphi}, e_\theta) = (e_{\theta \circ \varphi^4}, e_{\theta \circ \varphi^3} + e_{\theta \circ \varphi}, e_{\theta \circ \varphi^2}, ue_{\theta \circ \varphi}, e_\theta) X
	$$
	where 
	$$
	X = \begin{pmatrix}
	0 & 0 & 0 & 0 & 1 \\
	1 & 0 & 0 & 0 & 0 \\
	0 & 1 & 0 & 0 & 0 \\
	0 & 0 & 1 & 0 & -1 \\
	0 & 0 & 0 & 1 & 0
	\end{pmatrix} \begin{pmatrix}
	u^n & 0 & 0 & 0 & 0 \\
	0 & 1 & 0 & 0 & 0 \\
	0 & 0 & u^{n-1} & 0 & 0 \\
	0 & 0 & 0 & u^{p} & 0 \\
	0 & 0 & 0 & 0 & u^x
	\end{pmatrix}
	\begin{pmatrix}
	1 & 0 & 0 & 0 & 0 \\
	0 & 1 & 0 & 0 & 0 \\
	-1 & 0 & 1 & 0 & 0 \\
	0 & 0 & 0 & 1 & 0 \\
	0 & 0 & 0 & 0 & 1
	\end{pmatrix}
	$$
	which shows that $M \in \operatorname{Mod}^{\operatorname{SD}}_k(\mathcal{O})$.
	
	\subsection{Irreducibility and strong divisibility}\label{irredsub}
	Let $L/K$, $l/k$ and $L_\infty/K_\infty$ be as in Notation~\ref{restrictionUNRAM}; we obtain $f\colon \mathfrak{S} \rightarrow \mathfrak{S}_L$. Let $N \in \operatorname{Mod}_l^{\operatorname{SD}}(\mathcal{O})$ be the rank one object given by
	\begin{equation*}
	N = l[[u]] \otimes_{\mathbb{F}_p} \mathbb{F}, \qquad \varphi_N(1) = \sum_{\theta \in \operatorname{Hom}_{\mathbb{F}_p}(l,\mathbb{F})} u^{r_\theta} e_\theta
	\end{equation*}
	Since $N \in \operatorname{Mod}^{\operatorname{SD}}_k(\mathcal{O})$ each $r_\theta \in [0,p]$. Note this $N$ is as in Lemma~\ref{rankone}, except we've fixed $x = 1$. This is to simplify notation (it will be easy to reduce from the general case to this one). The following proposition describes which Breuil--Kisin modules embed into $f_*N$ as in Lemma~\ref{intermsofN}.
	
	\begin{proposition}\label{explicitSD}
	Assume $T(f_*N)$ is irreducible. Let $M \subset f_*N$ be a finite free $k[[u]] \otimes_{\mathbb{F}_p} \mathbb{F}$-sub-module with $M[\frac{1}{u}] = (f_*N)[\frac{1}{u}]$. Then $M \in \operatorname{Mod}^{\operatorname{SD}}_k(\mathcal{O})$ if and only if the following conditions are satisfied.
	\begin{enumerate}   
		\item If $m \in M$ then $\varphi(m) \in M$ and if $\varphi(m) \in u^{p+1}M$ then $m \in uM$.
		\item If $\sum \alpha_\theta e_\theta \in M$ with $\alpha_{\theta} \in \mathbb{F}$ then 
		$$
		\sum_{r_\theta \equiv r~\operatorname{ mod}p} \alpha_{\theta} e_{\theta} \in M
		$$
		for every $0 \leq r \leq  p$.
	\end{enumerate}
	\end{proposition}
	\begin{proof}[Proof that SD implies (1) and (2)]
		If $M \in \operatorname{Mod}^{\operatorname{SD}}_k(\mathcal{O})$ then $F^0 M_k = M_k$ and $F^{p+1}M_k = 0$. The first condition implies $\varphi(m) \in M$ whenever $m \in M$. The second implies any $m \in M$ with $\varphi(m) \in u^{p+1}M$ must be zero in $M_k$, and so is contained in $uM$. 
		
		Before we verify (2) we explain how (1) implies $ue_\theta \in M$ for every $\theta$. Since $M[\frac{1}{u}] = (f_*N)[\frac{1}{u}]$ there is, for each $\theta$, a smallest integer $\delta_{\theta} \geq 0$ with $u^{\delta_{\theta}} e_\theta \in M$. We have $\varphi(u^{\delta_{\theta \circ \varphi}}e_{\theta \circ \varphi}) = u^{\delta_{\theta \circ \varphi}p - \delta_{\theta} + r_\theta} u^{\delta_{\theta}}e_\theta$, so (1) implies $\delta_{\theta \circ \varphi}p - \delta_{\theta} + r_\theta \in [0,p]$. Therefore $\delta_{\theta \circ \varphi}p - \delta_{\theta} \leq p$ and
		$$
		(p^{[l:\mathbb{F}_p]}-1)\delta_\theta = \sum_{i=0}^{[l:\mathbb{F}_p]-1} p^i(p\delta_{\theta \circ \varphi^{i+1}} - \delta_{\theta \circ \varphi^i}) \leq p(p^{[l:\mathbb{F}_p]}-1)/(p-1)
		$$
		This implies $\delta_\theta \in [0,1]$ if $p>2$, and $\delta_\theta \in [0,2]$ if $p=2$. If $p=2$ and $\delta_{\theta \circ \varphi} =2$ then, as $r_\theta + p\delta_{\theta \circ \varphi} - \delta_\theta \in [0,p]$, we must have $\delta_\theta = 2$ and $r_\theta = 0$. Thus $r_\theta =0$ for all $\theta \in \operatorname{Hom}_{\mathbb{F}_p}(l,\mathbb{F})$ and so $T(N)$ is the trivial character. In this case $T(f_*N)$ is not irreducible.
		
		To prove (2) we first make the following claim. Suppose that $\sum \alpha_{\theta} e_\theta \in M$ with $\alpha_{\theta} \in \mathbb{F}[[u]]$ (so this sum is more general than that in (2)) and that $u^r \sum \alpha_{\theta} e_\theta \in M^\varphi$ for $r \geq 0$. Since $\operatorname{Weight}(M) \subset [0,p]$ we can assume that $r \leq p$. Then:
		\begin{itemize}
			\item There exist $\widetilde{\alpha_{\theta,r}} \in \mathbb{F}[[u]]$ such that $\sum \widetilde{\alpha_{\theta,r}} e_\theta \in M$, $u^{r-1} \sum \widetilde{\alpha_{\theta,r}} e_\theta \in M^\varphi$ and
			$$
			\widetilde{\alpha_{\theta,r}} \equiv \begin{cases} 
			\alpha_{\theta}~\operatorname{mod} u & \text{if $r_\theta \neq r$, except possibly if $r_\theta = 0$ and $r = p$} \\
			0~\operatorname{mod} u & \text{if $r_\theta  = r$}
			\end{cases}
			$$
		\end{itemize}
		To verify the claim we use that, since $M$ is strongly divisible, the map $M_k \rightarrow M^\varphi_k$ is an isomorphism of filtered modules. As $u^r\sum \alpha_{\theta} e_\theta \in F^rM^\varphi$ it follows that there exists an element $\beta \in F^r M$ such that $\varphi(\beta) - u^r\sum \alpha_{\theta} e_\theta \in uM^\varphi$. If $\beta = \sum \beta_\theta e_{\theta \circ \varphi}$ then 
		$$
		\sum \varphi(\beta_{\theta}) u^{r_\theta} e_\theta - u^r \sum \alpha_{\theta} e_\theta = \sum \left( \varphi(\beta_{\theta}) u^{r_\theta} - u^r\alpha_{\theta} \right) e_\theta \in uM^\varphi \cap u^r M
		$$
		As $u^r M \subset u^rN$ and $uM^\varphi \subset uN^\varphi$ we deduce that 
		$$
		v_u \left( \varphi(\beta_{\theta}) u^{r_\theta} - u^r\alpha_{\theta} \right) > \operatorname{max} \lbrace r_\theta,r-1\rbrace
		$$
		 Here $v_u$ denotes the $u$-adic valuation. If $r_\theta > r$ this implies $\alpha_\theta \equiv \varphi(\beta_\theta) \equiv 0$ modulo~$u$. If $r_\theta =r$ then it implies $\alpha_{\theta} \equiv \varphi(\beta_\theta)$ modulo~$u$. If $r> r_\theta$ it implies $\varphi(\beta) \equiv 0$ modulo~$u$. We can therefore find $\gamma_{\theta} \in \mathbb{F}[[u]]$ such that $\varphi(\beta_\theta) = u^p \gamma_{\theta}$ if $r \neq r_\theta$ and $\varphi(\beta_\theta) = \alpha_{\theta} + u^p \gamma_{\theta}$ if $r_\theta = r$. We conclude
		$$
		 \sum_{r_\theta \neq r} \alpha_{\theta} e_\theta - \sum u^{p -r + r_\theta}\gamma_{\theta} e_\theta \in M
		$$
		and $u^{r-1}$ times this element is contained in $M^\varphi$. Taking $\widetilde{\alpha_{\theta,r}} = \alpha_{\theta} -  u^{p-r +r_\theta} \gamma_{\theta}$ when $r_\theta \neq r$ and $\widetilde{\alpha_{\theta,r}} = u^{p-r+ r_\theta} \gamma_\theta$ when $r_\theta  = r$ gives the claim.
		
		We now use the claim to verify (2). Suppose $\sum \alpha_\theta e_\theta \in M$, now with $\alpha_{\theta} \in \mathbb{F}$. As already remarked, the fact that $\operatorname{Weight}(M) \subset [0,p]$ implies $u^pM \subset M^\varphi$. In particular $u^p\sum \alpha_{\theta} e_\theta \in M^\varphi$ so the claim applies, and produces $\sum \widetilde{\alpha_{\theta,p}} e_\theta \in M$. Using that $ue_\theta \in M$ for every $\theta$ we deduce that there are $\gamma_\theta \in \mathbb{F}$ such that $\sum_{r_\theta \neq p} \alpha_{\theta} e_\theta + \sum_{r_\theta = 0} \gamma_{\theta} e_\theta \in M$. Hence
		$$
		\sum_{r_\theta = p} \alpha_{\theta} e_\theta - \sum_{r_\theta = 0} \gamma_\theta e_\theta \in M
		$$
		As $u^{p-1} \sum \widetilde{\alpha_{\theta,p}} e_\theta \in M^\varphi$ we can then apply the claim to $ \sum \widetilde{\alpha_{\theta,p}} e_\theta$, this yields $\sum \widetilde{\alpha_{\theta,p-1}} e_\theta \in M$. Again using that $ue_\theta \in M$ for each $\theta$ we deduce that $\sum_{r_\theta \neq p,p-1} \alpha_{\theta} e_\theta + \sum_{r_\theta = 0} \gamma_{\theta} e_\theta \in M$, and hence
		$$
		\sum_{r_\theta = p-1} \alpha_{\theta} e_\theta \in M
		$$
		Repeatedly applying the claim in this fashion we deduce that $\sum_{r_\theta  = r} \alpha_{\theta} e_\theta$ for $0< r< p$ and $\sum_{r_\theta = 0} \alpha_{\theta} e_\theta + \sum_{r_\theta = 0} \gamma_{\theta} e_\theta \in M$. In particular we find
		$$
		\sum_{r_\theta = p} \alpha_{\theta} e_\theta + \sum_{r_\theta  = 0 } \alpha_{\theta} e_\theta = \sum_{r_\theta \equiv 0} \alpha_{\theta} e_\theta \in M
		$$ 
		which finishes the proof.
	\end{proof}
	
	\subsection{Finishing the proof of Proposition~\ref{explicitSD}}\label{subirred2} Let $N$ be as in the previous subsection and suppose that $M \subset f_*N$ is a free $k[[u]] \otimes_{\mathbb{F}_p} \mathbb{F}$-module with $M[\frac{1}{u}] = (f_*N)[\frac{1}{u}]$. Assume that $M$ satisfies conditions (1) and (2) from Proposition~\ref{explicitSD}. We are going to prove that $M \in \operatorname{Mod}^{\operatorname{SD}}_k(\mathcal{O})$. Along the way we shall describe the weights of $M$ in terms of the $r_\theta$.
	\begin{construction}
		For a fixed $\lambda \in \operatorname{Hom}_{\mathbb{F}_p}(l,\mathbb{F})$ define an ordering on $\operatorname{Hom}_{\mathbb{F}_p}(l,\mathbb{F})$ by asserting  that
		$$
		\lambda \circ \varphi <_\lambda \lambda \circ \varphi^2 <_\lambda \ldots <_{\lambda} \lambda \circ \varphi^{[l\colon \mathbb{F}_p] - 1} <_\lambda \lambda
		$$ 
		Using this ordering we define $X \subset \operatorname{Hom}_{\mathbb{F}_p}(l,\mathbb{F})$ by
		\begin{equation}\label{defnofX}
		\theta \not\in X \Leftrightarrow \text{there exists $\alpha_{\kappa} \in \mathbb{F}$ such that $e_\theta+\sum_{\kappa <_\lambda \theta} \alpha_{\kappa} e_{\kappa} \in M$}
		\end{equation}
		Clearly $X$ depends upon the choice of $\lambda$.
	\end{construction}
	\begin{lemma}\label{unique}
		If $\theta \not\in X$ there exists a unique $\mathbb{F}$-linear combination
		$$
		e_{\theta} + \sum \alpha_{\kappa} e_{\kappa} \in M, \qquad \alpha_{\kappa} \in \mathbb{F}
		$$
		in which the sum runs over $\kappa \in X$ satisfying (i) $\kappa <_\lambda \theta$ (ii) $r_{\kappa} \equiv r_\theta$ modulo~$p$ and (iii) $\kappa|_k = \theta|_k$. In particular the element lies in $M_{\theta|_k}$.
	\end{lemma}
\begin{proof}
	As $\theta \not\in X$, there exists $e_\theta + \sum \alpha_{\kappa} e_{\kappa} \in M$ with the sum running over $\kappa <_\lambda \theta$. Arguing inductively one shows there exists such a sum running only over those $\kappa <_\lambda \theta$ with $\kappa \in X$. There can be at most one sum of this form. To see this note that if there were two their difference would give a non-zero $\mathbb{F}$-linear combination $\sum_{\kappa \in X} \beta_{\kappa} e_{\kappa} \in M$. Then the maximal (with respect to $<_\lambda$) $\kappa$ with $\beta_{\kappa} \neq0$ would not be contained in $X$, a contradiction. Condition (2) of Proposition~\ref{explicitSD} therefore implies the sum may be taken to run over $\kappa$ additionally satisfying $(ii)$. As $M = \prod_{\tau \in \operatorname{Hom}_{\mathbb{F}_p}(k,\mathbb{F})} M_{\tau}$ we also have $(iii)$.
\end{proof}
\begin{definition}
	Consider $\mathbb{F}$-linear combinations of the form
	\begin{equation}\label{smallest}
	e_{\iota} + \sum_{0 < j \leq I} \alpha_j e_{\iota \circ \varphi^j} \in M
	\end{equation}
	 with $0 \leq I < [l:\mathbb{F}_p]$ and $\iota \in \operatorname{Hom}_{\mathbb{F}_p}(l,\mathbb{F})$. We say \eqref{smallest} is minimal if there exists no $\iota' \in \operatorname{Hom}_{\mathbb{F}_p}(l,\mathbb{F})$ together with an $\mathbb{F}$-linear combination $e_{\iota'} + \sum_{0 < j \leq J} \alpha_j e_{\iota' \circ \varphi^j} \in M$ such that $J<I$. Note that for a fixed $\iota$ there can exist at most one minimal sum as in \eqref{smallest}; if there were two their difference would have shorter length.
\end{definition}
\begin{lemma}\label{equal}
	If \eqref{smallest} is a minimal sum then $r_{\iota \circ \varphi^j} = r_\iota$ whenever $\alpha_j \neq 0$ and $j \leq I$.
\end{lemma}
\begin{proof}
	Uniqueness of minimal elements and condition (2) of Proposition~\ref{explicitSD} implies $r_{\iota} \equiv r_{\iota \circ \varphi^i}$ modulo~$p$. Since each $r_{\iota \circ \varphi^j} \in [0,p]$ this will be an equality, except possibly if $r_\iota = 0$ or $p$. In this case set 
	$$
	z =  u^{\gamma_0} e_{\iota \circ \varphi} + \sum_{0 <j \leq I} u^{\gamma_{j}} \alpha_{j} e_{\iota \circ \varphi^{j+1}}
	$$
	where $\gamma_{j} = 0$ if $r_{\iota \circ \varphi^j} = p$ and $\gamma_{j} = 1$ if $r_{\iota \circ \varphi^j} = 0$. Then $\varphi(z)$ equals $u^p$ times \eqref{smallest} and so condition (1) of Proposition~\ref{explicitSD} implies $z \in M$. Thus either all $\gamma_i =0$ or all equal $1$, otherwise we would obtain an element of $M$ contradicting the minimality of \eqref{smallest}.
\end{proof}	

The next proposition is where we use that $T(M)= T(f_*N)$ is irreducible.
\begin{proposition}\label{chooselambda}
	There exists $\lambda \in \operatorname{Hom}_{\mathbb{F}_p}(l,\mathbb{F})$ such that 
	\begin{enumerate}
		\item If $\theta \in X$ and $\theta \circ \varphi \not\in X$ then $r_\theta >0$.
		\item If $\theta \not\in X$ and $\theta \circ \varphi \in X$ then $r_{\theta} = 0$.
		\item If $\theta \in X$ and $e_{\theta \circ \varphi} \not\in M$ then $0 \leq r_\theta \leq 1$. In particular this holds if $\theta \circ \varphi \in X$.
	\end{enumerate}
\end{proposition}
\begin{proof}
	First we show (3) holds for any choice of $\lambda$. If $e_{\theta \circ \varphi} \not\in M$ then condition (1) of Proposition~\ref{explicitSD} implies $\varphi(ue_{\theta \circ \varphi}) = u^{r_\theta + p - 1} (ue_\theta) \not\in u^{p+1}M$. If $\theta \in X$ then $e_\theta \not\in M$ so $r_\theta + p -1 \leq p$, and $r_\theta \leq 1$.
	
	Next we show (2) holds whenever $r_\lambda =0$. Suppose $\theta \not\in X$ and $r_\theta > 0$ (we're assuming that $r_\lambda = 0$ so $\theta \neq \lambda$). We'll show $\theta \circ \varphi \not\in X$. Choose $e_\theta + \sum_{\kappa <_\lambda \theta} \alpha_{\kappa}e_\kappa \in M$ as in Lemma~\ref{unique}. Set $z = e_{\theta \circ \varphi} + \sum_{r_\kappa \neq 0} \alpha_{\kappa} e_{\kappa \circ \varphi} + u\sum_{r_\kappa = 0} \alpha_{\kappa} e_{\kappa \circ \varphi}$. Using that $r_\theta \equiv r_\kappa$ modulo~$p$ and $r_\theta >0$ we see that $\varphi(z) = u^{r_\theta}( e_\theta + \sum \alpha_{\kappa}e_\kappa)$. Condition (1) of Proposition~\ref{explicitSD} implies $z \in M$. Since $\theta \neq \lambda$, if $\kappa <_\lambda \theta$ then $\kappa \circ \varphi <_\lambda \theta \circ \varphi$. Therefore $z$ shows $\theta \circ \varphi \not\in X$.
	
	Now choose a minimal sum as in \eqref{smallest} (if none exists then we must have $M = u(f_*N)$ and so $X = \operatorname{Hom}_{\mathbb{F}_p}(l,\mathbb{F})$, in which case conditions (1), (2), and (3) hold vacuously). We are going to show that either $\lambda = \iota$ satisfies the conditions of the proposition or $e_{\iota \circ \varphi} + \sum \alpha_{j} e_{\iota \circ \varphi^{j+1}} \in M$. Let us explain why this implies the proposition. If there exists no $\lambda\in \operatorname{Hom}_{\mathbb{F}_p}(l,\mathbb{F})$ satisfying conditions (1)-(3) then it would follow that $e_{\iota \circ \varphi^i} + \sum \alpha_{j} e_{\iota \circ \varphi^{j+i}} \in M$ for every $i \geq 0$. Lemma~\ref{equal} then implies $r_{\iota'} = r_{\iota' \circ \varphi^I}$ for every $\iota' \in \operatorname{Hom}_{\mathbb{F}_p}(l,\mathbb{F})$. If $\omega$ is the character through which $G_{K}$ acts on $T(N)$ we then have $\omega = \prod_\iota \omega_{\iota}^{-r_\iota} = \prod \omega_{\iota \circ \varphi^I}^{-r_\iota} = \omega^{p^I}$. If $I >0$ this contradicts the irreducibility of $T(M) = \operatorname{Ind}_L^K T(N)$. If $I = 0$ it follows that every $e_{\iota'} \in M$, so $X = \emptyset$ and conditions (1)-(3) hold vacuously. 
	
	Set $z = e_{\iota \circ \varphi} + \sum \alpha_{j} e_{\iota \circ \varphi^{j+1}}$. If $r_{\iota} >0$ then $z \in M$ follows without any further assumptions; Lemma~\ref{equal} says that $r_{\iota \circ \varphi^j} = r_\iota$ whenever $\alpha_{\iota \circ \varphi^j} \neq 0$, and so $\varphi(z) \in u^{r_\iota}M$. Thus $z \in M$ by condition (1) of Proposition~\ref{explicitSD}. If instead $r_{\iota} = 0$ set $\lambda = \iota$. The first and second paragraph of this proof shows that (2) and (3) hold. Thus (1) cannot hold, and so there must exist $\theta \in X$ with $\theta \circ \varphi \not\in X$ and $r_{\theta} = 0$. We use this to show $z \in M$. If $\theta = \lambda$ then $\lambda \circ \varphi \not\in X$ which means $e_{\lambda \circ \varphi} \in M$; by minimality $z = e_{\lambda \circ \varphi}$ and we are done. Let us therefore assume $\theta \neq \lambda$. Consider the unique
	$$
	f_{\theta \circ \varphi} = e_{\theta \circ \varphi} + \sum_{\kappa <_{\lambda} \theta \circ \varphi} \alpha_{\kappa}e_\kappa
	$$
	from Lemma~\ref{unique}. As $r_\theta = 0$ and $\kappa <_{\lambda} \theta \circ \varphi$ implies $\kappa \circ \varphi^{-1} <_{\lambda} \theta$, except if $\kappa = \lambda \circ \varphi$, we obtain
	$$
	\varphi(f_{\theta \circ \varphi}) = e_\theta + \alpha_{\lambda \circ \varphi} e_\lambda +  \sum_{\kappa \circ \varphi^{-1}<_{\lambda} \theta} \alpha_{\kappa \circ \varphi} u^{r_{\kappa \circ \varphi^{-1}}} e_{\kappa \circ \varphi^{-1}} \in M
	$$
	Removing those terms with $r_{\kappa \circ \varphi^{-1}} > 0$ and re-indexing, we obtain 
	\begin{equation}\label{contradiction}
	e_\theta + \alpha e_{\lambda} + \sum_{\kappa <_{\lambda} \theta} \beta_{\kappa} e_{\kappa} \in M
	\end{equation}
	for some $\alpha, \beta_{\kappa} \in \mathbb{F}$. If $\alpha = 0$ then \eqref{contradiction} contradicts the assumption that $\theta \in X$. If we write $\theta = \lambda \circ \varphi^J$ and $J<I$ then \eqref{contradiction} contradicts the assumption that \eqref{smallest} is minimal. If $I < J$ then the different between \eqref{contradiction} and $\alpha$ multiplied by \eqref{smallest} again contradicts the assumption that $\theta \in X$. Thus $I = J$. The uniqueness of minimal elements then implies \eqref{contradiction} equals $\alpha$ times \eqref{smallest}. Thus $ z = \frac{f_{\theta \circ \varphi}}{\alpha} \in M$ which completes the proof.
 \end{proof}
		\begin{proof}[End of the proof of Proposition~\ref{explicitSD}]
			We have to show $M$ is strongly divisible. Fix $\lambda$ as in Proposition~\ref{chooselambda} and for $\theta \in \operatorname{Hom}_{\mathbb{F}_p}(l,\mathbb{F})$ set
			$$
			f_\theta = \begin{cases}
			e_{\theta} + \sum \alpha_{\kappa} e_{\kappa} \text{ as in Lemma~\ref{unique}} & \text{if $\theta \not\in X$} \\
			ue_{\theta} & \text{if $\theta \in X$} 
			\end{cases}
			$$
			For $\tau \in \operatorname{Hom}_{\mathbb{F}_p}(k,\mathbb{F})$ the $f_\theta$ with $\theta|_k = \tau$ form an $\mathbb{F}[[u]]$-basis of $M_\tau$. To see this let $W \subset M_\tau$ be the subspace they span. It is easy to see that if $\theta|_k = \tau$ then $ue_\theta \in W$. It therefore suffices to show any $\sum \alpha_\theta e_{\theta} \in M_\tau$ with $\alpha_{\theta} \in \mathbb{F}$ is in $W$. We see that $\sum \alpha_{\theta}e_\theta - \sum_{\theta \not\in X} \alpha_{\theta}f_\theta$ is an $\mathbb{F}$-linear combination of $e_\theta$ with $\theta \in X$, and is contained in $M$. Such a linear combination must be zero (cf. the proof of Lemma~\ref{unique}) so $W = M_\tau$, as claimed.
			
			For each $\theta$ we now construct elements $g_{\theta \circ \varphi} \in M_{\theta \circ \varphi|_k}, h_\theta \in M_{\theta|_k}$ so that $\varphi(g_{\theta \circ \varphi}) = u^{r_\theta + ps_{\theta \circ \varphi} - s_\theta} h_\theta$ where
			\begin{equation}\label{stheta}
			s_\theta =\begin{cases}
			1 & \text{if $\theta \in X$} \\
			0 & \text{if $\theta \notin X$}
			\end{cases}
			\end{equation}
			We do this on a case-by-case basis.
			\begin{itemize}
				\item Suppose $\theta \not\in X$ and $\theta \circ \varphi \in X$. Set $h_\theta := f_\theta = e_\theta + \sum_{\kappa <_\lambda \theta, \kappa \in X} \alpha_{\kappa} e_{\kappa}$. (2) of Proposition~\ref{chooselambda} implies $r_\theta = 0$, so each $r_\kappa$, being congruent to $r_\theta$ modulo~$u$, equals $0$ or $p$. If $r_\kappa = p$ then (3) of Proposition~\ref{chooselambda} implies $e_{\kappa \circ \varphi} \in M$, and so $\kappa \circ \varphi \not\in X$ and $f_{\kappa \circ \varphi} = e_{\kappa \circ \varphi}$. If $r_\kappa = 0$ then (1) of Proposition~\ref{chooselambda} implies $\kappa \circ \varphi \in X$. Thus
				$$
				g_{\theta \circ \varphi} := f_{\theta \circ \varphi} + \sum_{\kappa <_\lambda \theta, \kappa \in X} \alpha_{\kappa} f_{\kappa \circ \varphi} \in M_{\theta \circ \varphi|_k}
				$$
				is such that $\varphi(g_{\theta \circ \varphi}) = u^ph_\theta$.
				\item Suppose $\theta \not\in X$, $\theta \circ \varphi \not\in X$ and $r_\theta  =0$. Set $g_{\theta \circ \varphi} := f_{\theta \circ \varphi} = e_{\theta \circ \varphi} + \sum_{\kappa  <_\lambda \theta \circ \varphi,\kappa \in X} \alpha_{\kappa} e_{\kappa}$. Since $\kappa \in X$, if $\kappa \circ \varphi^{-1} \not\in X$ then $r_{\kappa \circ \varphi^{-1}} = 0$ by (2) of Proposition~\ref{chooselambda}. By (3) of Proposition~\ref{chooselambda}, if $\kappa \circ \varphi^{-1} \in X$ then $r_{\kappa \circ \varphi^{-1}} \in [0,1]$. Therefore the difference between $\varphi(g_{\theta \circ \varphi})$ and
				$$
				h_\theta := f_{\theta} + \sum_{\kappa <_\lambda \theta \circ \varphi,\kappa \circ \varphi^{-1} \not\in X} \alpha_{\kappa} f_{\kappa \circ \varphi^{-1}} + \sum_{\kappa <_\lambda \theta \circ \varphi,\kappa \circ \varphi \in X,r_{\kappa \circ \varphi^{-1}} = 1} \alpha_{\kappa} f_{\kappa \circ \varphi^{-1}}
				$$ 
				is an $\mathbb{F}$-linear combination of $e_\kappa$ with $\kappa \in X$. Since this $\mathbb{F}$-linear combination is contained in $M$ it must be zero (cf. the proof of Lemma~\ref{unique}). Therefore $\varphi(g_{\theta \circ \varphi}) = h_\theta$.
				\item Suppose $\theta \not\in X$, $\theta \circ \varphi \not\in X$ and $r_\theta  > 0$. Set $h_\theta := f_\theta = e_\theta + \sum_{\kappa <_\lambda \theta,\kappa \in X} \alpha_{\kappa}e_\kappa$. Each $r_\kappa \equiv r_\theta$ modulo~$p$ and so $\varphi$ sends
				$$
				e_{\theta \circ \varphi} + \sum_{r_\kappa > 0} \alpha_{\kappa}e_{\kappa \circ \varphi} + u\sum_{r_\kappa = 0} \alpha_{\kappa} e_{\kappa \circ \varphi}
				$$
				onto $u^{r_\theta} h_\theta$. As $r_\theta>0$ this displayed sum is contained in $M$ by condition (1) of Proposition~\ref{explicitSD}. We claim this displayed sum is equal to
				$$
				g_{\theta \circ \varphi} := f_{\theta \circ \varphi} + \sum_{\kappa <_\lambda \theta, \kappa \circ \varphi \not\in X} \alpha_{\kappa} f_{\kappa \circ \varphi} + \sum_{\kappa <_\lambda \theta, \kappa \circ \varphi \in X, r_\kappa  =0} \alpha_{\kappa} f_{\kappa \circ \varphi}
				$$
				To see this note that, by (1) of Proposition~\ref{chooselambda}, if $r_\kappa = 0$ then $\kappa \circ \varphi \in X$ and if $r_{\kappa \circ \varphi} \not\in X$ then $r_\kappa > 0$. From this it follows that the difference between these two sums, which is an element of $M$, is an $\mathbb{F}$-linear combination of $e_\kappa$ with $\kappa \in X$. This difference is therefore zero, and so $\varphi(g_{\theta \circ \varphi}) = u^{r_\theta} h_\theta$.
				\item Suppose $\theta \in X$ and $\theta \circ \varphi \not\in X$. Set $g_{\theta \circ \varphi} := f_{\theta \circ \varphi} = e_{\theta \circ \varphi} + \sum_{\kappa<_\lambda \theta \circ \varphi,\kappa \in X} \alpha_{\kappa}e_{\kappa}$, and set 
				$$
				h_\theta := f_\theta + \sum_{\kappa <_\lambda \theta \circ \varphi,\kappa \circ \varphi^{-1} \not\in X} \alpha_{\kappa} f_{\kappa \circ \varphi^{-1}} + \sum_{\kappa<_\lambda \theta \circ \varphi,\kappa \circ \varphi^{-1} \in X, r_{\kappa \circ \varphi^{-1}} = 1} \alpha_{\kappa} f_{\kappa \circ \varphi^{-1}}
				$$
				We claim $\varphi(g_{\theta \circ \varphi}) = u^{r_\theta -1}h_\theta$. If $e_{\theta \circ \varphi} \in M$ then this is clear since $g_{\theta \circ \varphi} = e_{\theta \circ \varphi}$ and $h_\theta = ue_\theta$. If $e_{\theta \circ \varphi} \not\in M$ then (1) and (3) of Proposition~\ref{chooselambda} implies $r_\theta = 1$, so we have to show $\varphi(g_{\theta \circ \varphi}) = h_\theta$. Proposition~\ref{chooselambda} tells us $\kappa \in X$ and $\kappa \circ \varphi^{-1} \not\in X$ implies $r_{\kappa \circ \varphi^{-1}} = 0$, while if $\kappa \in X$ and $\kappa \circ \varphi^{-1} \in X$ then $r_{\kappa \circ \varphi^{-1}} \in [0,1]$. Using these two facts we see that the difference between $\varphi(g_{\theta \circ \varphi})$ and $h_\theta$ is an $\mathbb{F}$-linear combination of $e_\kappa$ with $\kappa \in X$. Since this difference is contained in $M$ it must be zero.
				\item Finally, if $\theta \in X$ and $\theta \circ \varphi \in X$ set $g_{\theta \circ \varphi} := f_{\theta \circ \varphi}$ and $h_{\theta} := f_\theta$. Then $\varphi(g_{\theta \circ \varphi}) = u^{r_\theta + p-1} h_\theta$.
			\end{itemize}
		To finish the proof it suffices to show that for $\theta$ with $\theta|_k = \tau$, the $g_{\theta \circ \varphi}$ form an $\mathbb{F}[[u]]$-basis of $M_{\tau \circ \varphi}$, and the $h_\theta$ form an $\mathbb{F}[[u]]$-basis of $M_\tau$. If $H$ is the $\mathbb{F}[[u]]$-linear endomorphism of $M_\tau$ sending $f_\theta$ onto $h_\theta$ then $H - \operatorname{Id}$ sends $f_\theta$ onto $\mathbb{F}$-linear combinations of $f_{\kappa \circ \varphi^{-1}}$ with $\kappa <_\lambda \theta \circ \varphi$. Hence $H- \operatorname{Id}$ is nilpotent, $H$ is an automorphism, and the $h_\theta$ form an $\mathbb{F}[[u]]$-basis as claimed. A similar observation shows the $g_{\theta \circ \varphi}$ also form an $\mathbb{F}[[u]]$-basis.
		\end{proof}
		
		Using Remark~\ref{riremark} we deduce:
	\begin{corollary}\label{weights}
		With $s_\theta$ as in \eqref{stheta}
		$$
		\operatorname{Weight}_\tau(M) = \lbrace r_\theta + ps_{\theta \circ \varphi} - s_\theta \mid \theta|_k = \tau \rbrace
		$$
	\end{corollary} 
		\subsection{Putting everything together}
		 
		Applying what we've shown so far in this subsection gives:
		\begin{proposition}\label{main}
			Let $M \in \operatorname{Mod}^{\operatorname{SD}}_k(\mathcal{O})$ with $T(M)$ irreducible. Then there exist integers $\widetilde{r_{\theta}}$ indexed over $\theta \in \operatorname{Hom}_{\mathbb{F}_p}(l,\mathbb{F})$ such that (i):
			$$
			T(M) \otimes_{\mathbb{F}} \overline{\mathbb{F}}_p = \psi \otimes \operatorname{Ind}_{L_\infty}^{K_\infty} (\prod_{\theta} \omega_{\theta}^{-\widetilde{r_{\theta}}})
			$$
			for some unramified character $\psi$ and for $L_\infty = L(\pi^{1/p^\infty})$ with $L$ an unramified extension $K$, and such that (ii):
			$$
			\operatorname{Weight}_\tau(M) = \lbrace \widetilde{r_\theta} \mid \theta|_k = \tau \rbrace
			$$
		\end{proposition}
	\begin{proof}
		Lemma~\ref{intermsofN} produces a rank one $N \in \operatorname{Mod}^{\operatorname{SD}}_k(\mathcal{O})$, which we assume is as in Lemma~\ref{rankone}, together with an embedding $M \hookrightarrow f_*N$. We want to apply the results of Subsections~\ref{irredsub} and~\ref{subirred2}, so we require the $x \in \mathbb{F}^\times$ appearing in the definition of $N$ to be $1$. Let us explain how to reduce to this case. Let $\operatorname{ur}_x \in \operatorname{Mod}^{\operatorname{SD}}_k(\mathcal{O})$ be the rank one object given by
		$$
		\operatorname{ur}_x = k[[u]] \otimes_{\mathbb{F}_p} \mathbb{F}, \qquad \varphi_{\operatorname{ur}_x}(1) = x e_{\tau_0} + \sum_{\tau \neq \tau_0} e_\tau
		$$
		Set $\widetilde{M} = \operatorname{Hom}(\operatorname{ur}_x,M)^{\mathcal{O}}$ (recall Construction~\ref{Ohom}). One easily checks that $\widetilde{M} \in \operatorname{Mod}^{\operatorname{SD}}_k(\mathcal{O})$ and that $\operatorname{Weight}_\tau(\widetilde{M}) = \operatorname{Weight}_\tau(M)$ for each $\tau$ by verifying that condition (2) of Lemma~\ref{equivSDO} holds. The last sentence of Construction~\ref{Ohom} implies
		$$
		T(\widetilde{M}) = \operatorname{Hom}(\psi_x,\operatorname{Ind}_{L_\infty}^{K_\infty} \chi) = \operatorname{Ind}_{L_\infty}^{K_\infty} (\psi_x^{-1} \chi)
		$$
		Thus if the proposition holds for $\widetilde{M}$ it holds for $M$. Thus, assuming $x =1$ so that we can apply Corollary~\ref{weights}, we have $\operatorname{Weight}_\tau(M) = \lbrace r_\theta + ps_{\theta \circ \varphi} - s_\theta \mid \theta|_k = \tau \rbrace$. On the other hand we have $\chi = T(N)$ which is equal to $\prod_\theta \omega_{\theta}^{r_\theta +ps_{\theta \circ \varphi} - s_\theta}$ by Proposition~\ref{rankonesapplyT}. Therefore take $\widetilde{r_{\theta}} = r_\theta + ps_{\theta \circ \varphi} - s_\theta$.
	\end{proof}
\section{Crystalline representations}

In this section we state the key results which relate $\operatorname{Mod}^{\operatorname{SD}}_k(\mathcal{O})$ with crystalline representations. We then give a proof of the theorem from the introduction 
\subsection{Crystalline representations and Breuil--Kisin modules}
As in \cite{Fon94} let $B_{\operatorname{dR}}$ denote Fontaine's ring of $p$-adic periods, and $B_{\operatorname{crys}} \subset B_{\operatorname{dR}}$ the ring of crystalline periods. As in \cite{Fon94b} a $p$-adic representation $V$ of $G_K$ is crystalline if
$$
D_{\operatorname{crys}}(V) := (V \otimes_{\mathbb{Q}_p} B_{\operatorname{crys}})^{G_K}
$$
has $K_0$-dimension equal to $\operatorname{dim}_{\mathbb{Q}_p} V$. The inclusion $B_{\operatorname{crys}} \otimes_{K_0} K \subset B_{\operatorname{dR}}$ induces an equality $D_{\operatorname{crys}}(V)_K := D_{\operatorname{crys}}(V) \otimes_{K_0} K = (V \otimes_{\mathbb{Q}_p} B_{\operatorname{dR}})^{G_K}$ which allows us to equip $D_{\operatorname{crys}}(V)_K$ with the filtration
$$
F^i D_{\operatorname{crys}}(V)_K := (V \otimes_{\mathbb{Q}_p} t^i B_{\operatorname{dR}}^+)^{G_K}
$$ 
Here $B_{\operatorname{dR}}^+ \subset B_{\operatorname{dR}}$ is the discrete valuation ring with field of fractions $B_{\operatorname{dR}}$, and $t$ is any choice of uniformiser.
\begin{theorem}[Kisin]\label{Kisinthm}
	There is a fully faithful functor $T \mapsto M(T)$ which sends a crystalline $\mathbb{Z}_p$-lattice onto an object of $\operatorname{Mod}^{\operatorname{BK}}_K$ which is free over $\mathfrak{S}$. The Breuil--Kisin module $M(T)$ is uniquely determined by the fact that $T(M(T)) = T|_{G_{K_\infty}}$. 
\end{theorem}
\begin{proof}
	This is the main result of \cite{Kis06}. The formulation we give here is taken from \cite[Theorem 4.4]{BMS}.
\end{proof}
\begin{notation}
	A crystalline $\mathcal{O}$-lattice is a $G_K$-stable $\mathcal{O}$-lattice inside a continuous representation of $G_K$ on a finite dimensional $E$-vector space which is crystalline when viewed as a $\mathbb{Q}_p$-representation. By functoriality $M \mapsto T(M)$ restricts to a functor from the category of crystalline $\mathcal{O}$-lattices into $\operatorname{Mod}^{\operatorname{BK}}_K(\mathcal{O})$.
\end{notation}
\begin{definition}
	If $V$ is a crystalline representation on an $E$-vector space then $D_{\operatorname{crys}}(V)$ is a free module over $K_0 \otimes_{\mathbb{Q}_p} E$ of rank $\operatorname{dim}_E V$ and so $D_{\operatorname{crys}}(V)_K$ is a free $K_0 \otimes_{\mathbb{Q}_p} E$-module of rank $e\operatorname{dim}_E V$. If $K_0 \subset E$ then as in Construction~\ref{O-semilinear} there is a decomposition
	$$
	D_{\operatorname{crys}}(V)_K = \prod_{\tau\in \operatorname{Hom}_{\mathbb{F}_p}(k,\mathbb{F})} D_{\operatorname{crys}}(V)_{K,\tau}
	$$
	with each $D_{\operatorname{crys}}(V)_{K,\tau}$ a filtered $E$-vector space of dimension $e\operatorname{dim}_EV$. Define the $\tau$-th Hodge--Tate weights of $V$ to be the multiset $\operatorname{HT}_\tau(V)$ which contains $i$ with multiplicity 
	$$
	\operatorname{dim}_E \operatorname{gr}^i(D_{\operatorname{crys}}(V)_{K,\tau})
	$$
	With these normalisations the cyclotomic character has $\tau$-th Hodge--Tate weights $\lbrace -1,\ldots,-1 \rbrace$ ($e$ copies of $-1$).
\end{definition}

We need the following result of Gee--Liu--Savitt. 
\begin{theorem}[Gee--Liu--Savitt, Wang]\label{GLSthm2}
	Suppose $K = K_0$. If $p=2$ choose $\pi$ so that $K_\infty \cap K(\mu_{p^\infty}) = K$. If $T$ is a crystalline $\mathcal{O}$-lattice such that $\operatorname{HT}_\tau(V) \subset [0,p]$ where $V = T \otimes_\mathcal{O} E$, then $\overline{M} := M(T) \otimes_{\mathcal{O}} \mathbb{F} \in \operatorname{Mod}^{\operatorname{SD}}_k(\mathcal{O})$ and $\operatorname{Weight}_\tau(\overline{M}) = \operatorname{HT}_\tau(V)$.
\end{theorem}
\begin{proof}
	When $p>2$ this follows by reducing the description of $M(T)$ given in \cite[Theorem 4.22]{GLS} modulo any uniformiser of $\mathcal{O}$. The case $p=2$ follows similarly using \cite[Theorem 4.2]{Wang17} (note that the existence of a $\pi$ as stated is proven in \cite[Lemma 2.1]{Wang17}).\footnote{It is important when referencing both \cite{GLS} and \cite{Wang17} to keep track of differences in normalisation. In both these references $G_{K_\infty}$-representations are attached contravariantly to Breuil--Kisin modules and their Hodge--Tate weights are normalised to be the negative of ours.}  
\end{proof}
\subsection{Proof of main theorem} In this subsection we assume $K = K_0$. We can now give the proof of the theorem in the introduction. Recall that if $\overline{\rho}:G_K \rightarrow \operatorname{GL}_n(\overline{\mathbb{F}}_p)$ is a continuous representation then in Definition~\ref{defofWinert} we defined the set $\operatorname{Inert}(\overline{\rho})$.
\begin{theorem}\label{mainthm}
Let $K = K_0$. Let $\rho: G_K \rightarrow \operatorname{GL}_n(\overline{\mathbb{Z}}_p)$ be crystalline and suppose that $\operatorname{HT}_\tau(\rho) = (\lambda_{1,\tau} \leq \ldots \leq \lambda_{n,\tau})$ with $\lambda_{n,\tau} - \lambda_{1,\tau} \leq p$. Then
$$
(\lambda_{\tau}) \in \operatorname{Inert}(\overline{\rho})
$$ 
\end{theorem} 
\begin{proof}
Choose a coefficient field $E$ so that $\rho$ is defined over $\mathcal{O}$. Via a straightforward twisting argument we may suppose $\operatorname{HT}_\tau(\rho) \in [0,p]$. Let $M(\rho) \in \operatorname{Mod}^{\operatorname{BK}}_K(\mathcal{O})$ be the associated Breuil--Kisin module. By Theorem~\ref{GLSthm2}, $\overline{M} = M(\rho) \otimes_{\mathcal{O}} \mathbb{F} \in \operatorname{Mod}^{\operatorname{SD}}_k(\mathcal{O})$ and $\operatorname{HT}_\tau(\rho) = \operatorname{Weight}_\tau(\overline{M})$.

Choose a $G_K$-composition series of $\rho \otimes_{\mathcal{O}} \mathbb{F}$. Enlarging $E$ if necessary we can suppose that Lemma~\ref{irred} holds for each Jordan--Holder factor. Let $0 = \overline{M}_n \subset \ldots \subset \overline{M}_0 = \overline{M}$ be the corresponding composition series of $\overline{M}$. By Proposition~\ref{SDsubquotO} each of $\overline{M}_i/\overline{M}_{i+1} \in \operatorname{Mod}^{\operatorname{SD}}_k(\mathcal{O})$ and $\operatorname{Weight}_{\tau}(\overline{M}) = \bigcup_i \operatorname{Weight}_\tau(M_i/M_{i+1})$. Lemma~\ref{restricttoGKinfty} implies $T(\overline{M}_i/\overline{M}_{i+1})$ is induced from a character $\chi_i \colon LK_\infty \rightarrow \mathbb{F}^\times$ for some unramified extension $L/K$ (depending on $i$). Therefore  Proposition~\ref{main} applies to $\overline{M}_i/\overline{M}_{i+1}$ and shows that the weights of $\overline{M}_i/\overline{M}_{i+1}$ are contained in $\operatorname{Inert}(\chi_i)$. Since this is true for each $i$ we deduce $(\lambda_\tau) \in \operatorname{Inert}(\overline{\rho})$.
\end{proof}


\begin{bibdiv}
	\begin{biblist}
		\bib{BLGG}{article}{
			author={Barnet-Lamb, Thomas},
			author={Gee, Toby},
			author={Geraghty, David},
			title={Serre weights for $U(n)$},
			journal={J. Reine Angew. Math.},
			volume={735},
			date={2018},
			pages={199--224},
		}
		
		\bib{BLGGT}{article}{
			author={Barnet-Lamb, Thomas},
			author={Gee, Toby},
			author={Geraghty, David},
			author={Taylor, Richard},
			title={Potential automorphy and change of weight},
			journal={Ann. of Math. (2)},
			volume={179},
			date={2014},
			number={2},
			pages={501--609},
		}
		\bib{Berger11}{article}{
			author={Berger, Laurent},
			title={La correspondance de Langlands locale $p$-adique pour ${\rm
					GL}_2({\bf Q}_p)$},
			language={French, with French summary},
			note={S\'{e}minaire Bourbaki. Vol. 2009/2010. Expos\'{e}s 1012--1026},
			journal={Ast\'{e}risque},
			number={339},
			date={2011},
			pages={Exp. No. 1017, viii, 157--180},
		}
		\bib{BMS}{article}{
			author = {Bhatt, Bhargav}
			author = {Morrow, Matthew}
			author = {Scholze, Peter},
			title = {Integral {$p$}-adic {H}odge theory},
			journal = {Preprint, arXiv:1602.03148},
			date = {2016},
		}
		\bib{BourbakiCA}{book}{
			author={Bourbaki, Nicolas},
			title={Commutative algebra. Chapters 1--7},
			series={Elements of Mathematics (Berlin)},
			note={Translated from the French;
				Reprint of the 1972 edition},
			publisher={Springer-Verlag, Berlin},
			date={1989},
			pages={xxiv+625},
		}
		
		\bib{CH13}{article}{
			author={Chenevier, Ga\"{e}tan},
			author={Harris, Michael},
			title={Construction of automorphic Galois representations, II},
			journal={Camb. J. Math.},
			volume={1},
			date={2013},
			number={1},
			pages={53--73},
		}
		\bib{Col98}{article}{
			author={Colmez, Pierre},
			title={Th\'{e}orie d'Iwasawa des repr\'{e}sentations de de Rham d'un corps local},
			language={French},
			journal={Ann. of Math. (2)},
			volume={148},
			date={1998},
			number={2},
			pages={485--571},
		}
		
		
		\bib{Fon94}{article}{
			author={Fontaine, Jean-Marc},
			title={Le corps des p\'{e}riodes $p$-adiques},
			language={French},
			note={With an appendix by Pierre Colmez;
				P\'{e}riodes $p$-adiques (Bures-sur-Yvette, 1988)},
			journal={Ast\'{e}risque},
			number={223},
			date={1994},
			pages={59--111},
		}
		
		\bib{Fon94b}{article}{
			author={Fontaine, Jean-Marc},
			title={Repr\'{e}sentations $p$-adiques semi-stables},
			language={French},
			note={With an appendix by Pierre Colmez;
				P\'{e}riodes $p$-adiques (Bures-sur-Yvette, 1988)},
			journal={Ast\'{e}risque},
			number={223},
			date={1994},
			pages={113--184},
		}
		
		\bib{Fon00}{article}{
			author={Fontaine, Jean-Marc},
			title={Repr\'{e}sentations $p$-adiques des corps locaux. I},
			language={French},
			conference={
				title={The Grothendieck Festschrift, Vol. II},
			},
			book={
				series={Progr. Math.},
				volume={87},
				publisher={Birkh\"{a}user Boston, Boston, MA},
			},
			date={1990},
			pages={249--309},
		}
		\bib{GLS}{article}{
			author={Gee, Toby},
			author={Liu, Tong},
			author={Savitt, David},
			title={The Buzzard-Diamond-Jarvis conjecture for unitary groups},
			journal={J. Amer. Math. Soc.},
			volume={27},
			date={2014},
			number={2},
			pages={389--435},
		}
		\bib{Kis06}{article}{
			author={Kisin, Mark},
			title={Crystalline representations and $F$-crystals},
			conference={
				title={Algebraic geometry and number theory},
			},
			book={
				series={Progr. Math.},
				volume={253},
				publisher={Birkh\"{a}user Boston, Boston, MA},
			},
			date={2006},
			pages={459--496},
		}
		\bib{Liu10b}{article}{
			author={Liu, Tong},
			title={The correspondence between Barsotti-Tate groups and Kisin modules
				when $p=2$},
			language={English, with English and French summaries},
			journal={J. Th\'{e}or. Nombres Bordeaux},
			volume={25},
			date={2013},
			number={3},
			pages={661--676},
		}
		\bib{Sch12}{article}{
			author={Scholze, Peter},
			title={Perfectoid spaces},
			journal={Publ. Math. Inst. Hautes \'{E}tudes Sci.},
			volume={116},
			date={2012},
			pages={245--313},
			issn={0073-8301},
		}
		\bib{Serre72}{article}{
			author={Serre, Jean-Pierre},
			title={Propri\'{e}t\'{e}s galoisiennes des points d'ordre fini des courbes
				elliptiques},
			language={French},
			journal={Invent. Math.},
			volume={15},
			date={1972},
			number={4},
			pages={259--331},
		}
		\bib{SerreLA}{book}{
			author={Serre, Jean-Pierre},
			title={Local algebra},
			series={Springer Monographs in Mathematics},
			note={Translated from the French by CheeWhye Chin and revised by the
				author},
			publisher={Springer-Verlag, Berlin},
			date={2000},
			pages={xiv+128},
		}
		\bib{Wang17}{article}{
			author = {Wang, Xiyuan},
			title = {Weight elimination in dimensions 2 when $p=2$},
			journal = {Preprint, arXiv:1711.09035},
			date = {2017},
		}
	\end{biblist}
\end{bibdiv}
	\end{document}